\documentclass{article}
\usepackage{times}
\usepackage{amsmath,amsfonts,amsthm}
\usepackage{graphicx,color}
\usepackage{comment}
\usepackage{fullpage}
\usepackage{cite}

\newtheorem{proposition}{Proposition}
\newtheorem{convention}[proposition]{Convention}
\newtheorem{theorem}[proposition]{Theorem}
\newtheorem{lemma}[proposition]{Lemma}

\newtheorem{algorithm}[proposition]{Algorithm}
\newtheorem{definition}[proposition]{Definition}
\newtheorem{remark}[proposition]{Remark}

\newcommand{\R}{{\mathbb R}}
\newcommand{\dealii}{\textsc{deal.II}}
\newcommand{\parallelset}{\Pi}

\newcommand{\CellPar}[1]{\mathrel{\|_{#1}}}  
\newcommand{\LocPar}{\mathrel{\|_*}}  
\newcommand{\GlobPar}{\mathrel{\|}}  

\newcommand{\tria}{{\mathbb T}} 

\excludecomment{includefaces}

\begin{document}
\markboth{Agelek, Anderson, Bangerth, Barth}{Orienting edges of unstructured meshes}
\title{On orienting edges of unstructured two- and three-dimensional meshes}

\author{
  RAINER AGELEK\\
  Heidelberg
  \and
  MICHAEL ANDERSON\\
  DownUnder GeoSolutions
  \and
  WOLFGANG BANGERTH${}^\ast$\\
  Colorado State University
  \and
  WILLIAM L. BARTH\\
  University of Texas at Austin
}

\maketitle

\begin{abstract}
  Finite element codes typically use data structures that
  represent unstructured meshes as collections of cells, faces,
  and edges, each of which require associated coordinate
  systems. One then needs to store how the coordinate system of
  each edge relates to that of neighboring cells. On the other
  hand, we can simplify data structures and algorithms if we can
  a priori orient coordinate systems in such a way that the
  coordinate systems on the edges follow uniquely from those on
  the cells \textit{by rule}.

  Such rules require that \textit{every} unstructured mesh allows
  the assignment of directions to edges that satisfy the convention in
  adjacent cells. We show that the convention chosen for
  unstructured quadrilateral meshes in the \dealii{} library always allows
  to orient meshes. It can therefore be used
  to make codes simpler, faster, and less bug prone. We
  present an algorithm that orients meshes in $O(N)$
  operations. We then show that consistent orientations are not always possible
  for 3d hexahedral meshes. Thus, cells generally
  need to store the direction of adjacent edges, but our approach also allows the
  characterization of cases where this is not
  necessary. The 3d extension of our algorithm
  either orients edges consistently, or aborts, both
  within $O(N)$ steps.
\end{abstract}

\footnotetext{
  ${}^\ast$Corresponding author.\\
  Authors' addresses:\\
  R. Agelek: Heidelberg, Germany.

  M. Anderson:
  DownUnder GeoSolutions, Perth, Australia.
  \texttt{michaela@dugeo.com}

  W.~Bangerth:
  Department of Mathematics, Colorado State
  University, Fort Collins, CO 80523, USA.
  \texttt{bangerth@colostate.edu}.

  W. L. Barth:
  Texas Advanced Computing Center,
  The University of Texas at Austin, Austin, TX 78758, USA.
  \texttt{bbarth@tacc.utexas.edu}

  W.~Bangerth was partially supported by the National Science
  Foundation under award OCI-1148116 as part of the Software Infrastructure for
  Sustained Innovation (SI2) program; by the Computational
  Infrastructure in Geodynamics initiative (CIG), through the National Science
  Foundation under Award No.~EAR-0949446 and The University of California --
  Davis.
}

\section{Introduction}
\label{sec:introduction}

In most of the common numerical methods for the solution of partial
differential equations on bounded domains $\Omega\subset \R^d, d=2,3$,
one defines approximate solutions by first subdividing $\Omega$ into a
finite number of cells, and then setting up a system of equations on
these cells. Usually, cells are either triangular or quadrilateral
(for $d=2$), or tetrahedral, prismatic, pyramidal, or hexahedral (for
$d=3$). Because certain aspects of the solution may be associated with
cells or edges (in 2d), or cells, faces and edges (in 3d),
essentially all sufficiently general codes use data structures for such meshes that
explicitly or implicitly store not only cells and vertex locations,
but also faces and edges and allow associating data with these objects.

In many cases, the data that is associated with a cell, face, edge, or
vertex may have a physical location somewhere on this object. For
example, when using a ${\mathbb Q}_3$ bicubic finite element on a
rectangular cell, we need to store the index (and possibly the value)
of one degree of freedom per vertex, two along the edge (typically at
1/3 and 2/3 along its extent), and 4 inside the cell. In order to define where
a distance of 1/3 or 2/3 along the edge is, we need to define a coordinate
system on the edge. The same is true when implementing N\'ed\'elec elements
that define degrees of freedom as tangential vectors along edges, and
therefore need to define a direction on each edge. For similar
reasons, we typically also need coordinate systems within each cell.

We then need to define how the coordinate
system defined on the edge relates to that of the adjacent cells. We can
either do this by letting every 
quadrilateral cell store pointers to the four edges along with one bit per
edge that indicates how the direction of the edge embeds in the coordinate system of the
cell. Or we could seek a convention by which we orient all edges of the mesh once at the
beginning so that the orientation of each cell implies the orientation of its
bounding edges. In the latter case, we would not need to store direction
flags, and algorithms built on this convention would not need to provision for
different possible directions, thereby greatly simplifying code construction
and maintenance.

This paper is concerned with the following two questions:
\begin{itemize}
\item Is it possible to
  find such a convention for quadrilateral meshes
  (an example of such a mesh is shown in Fig.~\ref{fig:wing})? We will
  constructively show that this is indeed possible in 2d when adopting
  the convention that the edges that bound opposite sides of each cell
  point in the same direction; see the left panel of
  Fig.~\ref{fig:convention}.%
\footnote{On the other hand, the only two other reasonable conventions for edge
  directions, namely either requiring them to be (i) oriented in clockwise or
  counter-clockwise direction around the cell, or (ii) oriented away
  from two, oppositely located vertices, both do not always allow for
  globally unique directions of all edges of a mesh; see
  Remark~\ref{rem:circular-edges}.}

\item Is it possible to assign directions to all edges of a mesh that satisfy
  this convention in a computationally efficient manner? We will demonstrate
  that this is in fact so: Our construction of a proof for the answer to the
  first question (in
  Section~\ref{sec:2d}) also implies an algorithm that we show to be order
  optimal, i.e., it runs in a time proportional to the number of edges
  in the mesh.
\end{itemize}

\begin{figure}[tbp]
  \begin{center}
    \phantom{.}
    \hfill
    \includegraphics[height=.3\textwidth]{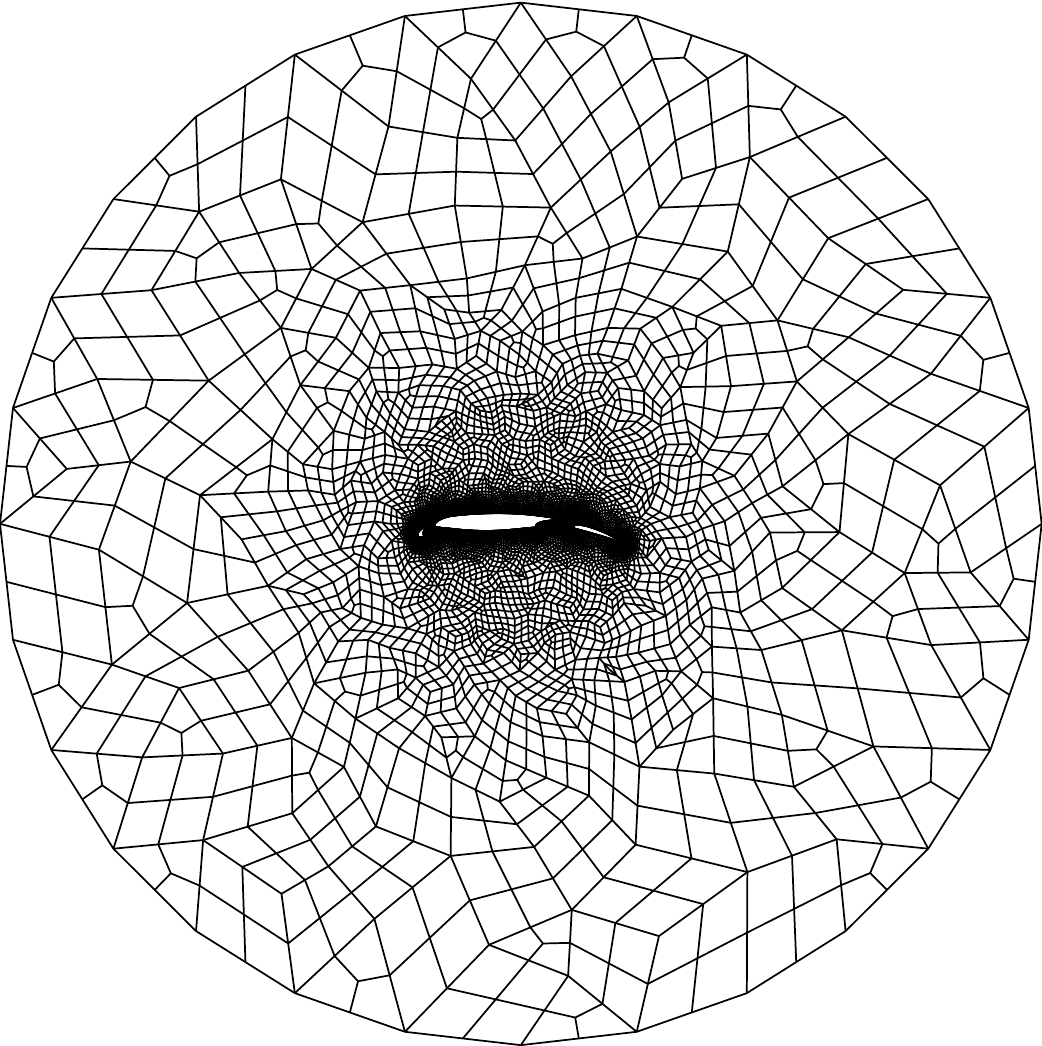}
    \hfill
    \includegraphics[height=.225\textwidth]{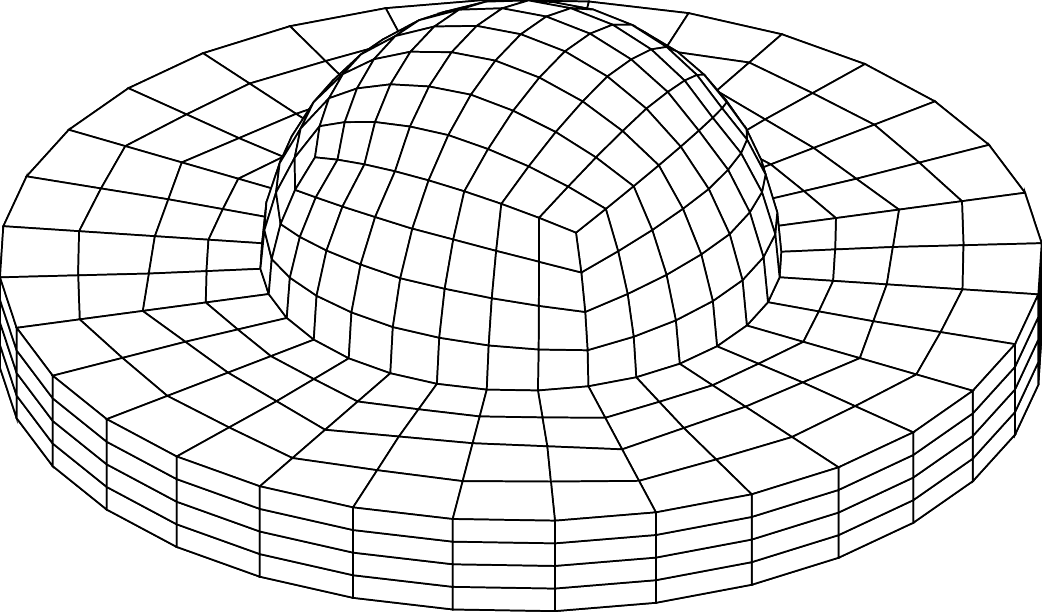}
    \hfill
    \phantom{.}
    \caption{\it A typical two-dimensional quadrilateral mesh around an
      airfoil with 29,490 cells (left). Surface of a
      three-dimensional mesh with 2,304 cells (right).}
    \label{fig:wing}
  \end{center}
\end{figure}

\begin{figure}[tbp]
  \begin{center}
    \phantom{.}
    \hfill
    \includegraphics[width=.2\textwidth]{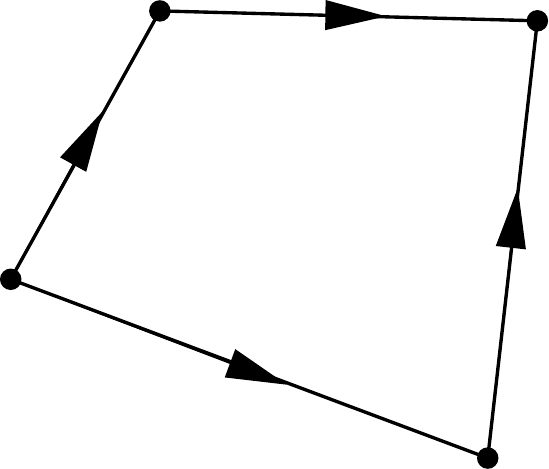}
    \hfill
    \includegraphics[width=.26\textwidth]{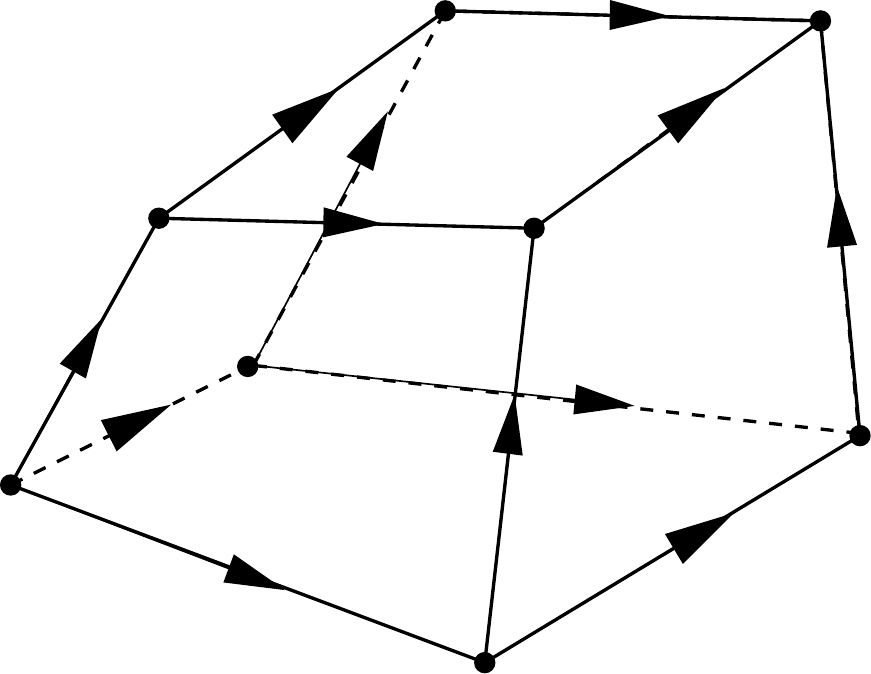}
    \hfill
    \phantom{.}
    \caption{\it Choice of directions of edges when seen with regard
      to one particular orientation of a ``coordinate system'' on a
      cell. Note in particular that this choice implies that opposite
      edges in a cell must have parallel directions. Left: For quadrilaterals in $d=2$.
      Right: For hexahedra in $d=3$.}
    \label{fig:convention}
  \end{center}
\end{figure}

On the other hand, we will show in Section~\ref{sec:3d} that the
corresponding convention in 3d (see the right 
panel of Fig.~\ref{fig:convention}) allows for examples
where it is not possible to orient edges so that coordinate systems of
adjacent cells are implied. However, we will show that the
extension of the 2d algorithm to 3d either produces a
consistent set of edge orientations or fails, both within order optimal
complexity. There are, however, important classes of 3d meshes that
always allow such edge directions, and we will discuss these in
Section~\ref{sec:3d-always-orientable}. 
\begin{includefaces}
  On the other hand, in 3d one may want to
  not only have consistent directions for edges of cells, but also for
  faces, and the corresponding case is discussed in
  Section~\ref{sec:3d-faces} where we show that there are cases where it is
  possible to find either consistent edge or face orientations, but not
  both.
\end{includefaces}

The paper is complemented by Section~\ref{sec:notation} defining
notation, conclusions in Section~\ref{sec:conclusions}, and an
appendix in which we we prove a generalization of the main
statement of the paper to general manifolds.

\paragraph*{Related literature}
Finite element software packages have traditionally taken different routes to
dealing with the problem of relative orientations of cells and their edges and
faces. In many cases, software has been developed to only support linear or
quadratic $H^1$-type elements, in which case edge and face orientations are not
in fact of any concern at all. Others use triangular or tetrahedral meshes for
which it is necessary to explicitly store edge orientations; see, for example,
\cite{BM02,AC03}. For quadrilaterals and hexahedra, 
strategies for implementation in specific packages are discussed in \cite{RKL09} for the FEniCS
library and \cite{TL08} for the KARDOS package. We have found that DUNE \cite{PBDEKKOS08} and Nektar++ \cite{Can15} appear to
explicitly store edge directions as part of their mesh data structures or
finite element implementations, but we could not find written elaborations of
their strategies in publications or overview documents. Finally, descriptions
of finite 
elements that use \textit{global} conventions for edge orientations
(i.e., based on vertex coordinates or indices, instead of in relation
to locally adjacent cells)
can be found in \cite{Zag06} and are used,
for example, in libMesh
\cite{KPSC07}.

Constructions similar to those in this paper have previously been discussed in
the discrete geometry literature, see \cite{Het95,AZ04,HW08} for examples. However,
this part of the literature is not typically concerned with algorithms and
their complexity (such as our discussions in Sections~\ref{sec:2d} and
\ref{sec:3d}), nor with the particular application of these ideas to finite
element meshes (such as our discussion of specific types of meshes
in Section~\ref{sec:3d-always-orientable}). Our contribution therefore provides a
relevant extension of what is available in the literature.

\paragraph*{A historical note} The algorithms discussed herein were
implemented in the \dealii{} library in 2003, with an incomplete
discussion of the topic available in the documentation of
\dealii{}'s \texttt{GridReordering} class. A more formal description of these
algorithms has recently appeared in \cite{HH15}; it extends our 2d algorithms to
meshes stored on distributed memory, parallel machines, but it lacks the
complexity analysis that we provide here, as well as much of the discussion of
the 3d case.

The introduction in \dealii{} of the convention discussed above predates 2003.
In its earliest days, the library was almost always used on small
problems for which edge orientations could be determined by hand on a
piece of paper, and little consideration was given to the
question whether it always exists and if so whether there is an efficient
algorithm to generate it automatically. However, as the project grew
and more applications used the 3d part, these questions became more
important. Initially, an algorithm that generated edge orientations using a
backtracking algorithm was implemented. This works for meshes
with up to a few hundred cells, but fails due to excessive run times for
larger ones. In particular, it is easy to construct meshes for which it had
exponential run time. Therefore, more efficient algorithms were needed,
leading to the results reported here.

\paragraph*{Availability of implementations} We have implemented the
algorithms outlined here in the \dealii{} library (see
\url{http://www.dealii.org/} and \cite{BHK07,dealII82}), and they are
available as part of the \texttt{GridReordering} class under the LGPL open source
license.

\section{Notation and conventions}
\label{sec:notation}

Throughout this paper, we will consider triangulations $\tria$ such as
those shown in Fig.~\ref{fig:wing}, as a collection
$\tria=\{K_1,\ldots,K_{N_K}\}$ of quadrilateral or hexahedral cells
$K_i$. These cells can be considered as open geometric
objects $K_i\subset \R^d, d=2,3$ so that (i)~$K_i\cap K_j=\emptyset$
if $i\neq j$, (ii)~the intersection of the closure of two cells, $\bar
K_i \cap \bar K_j$, is either empty, a vertex of the mesh, or a
complete edge or face of both cells, and (iii)~$\bigcup_i \bar K_i =
\bar\Omega$ where $\Omega\subset \R^d$ is the bounded, open domain that is
subdivided into the triangulation. We assume that the triangulation
has only finitely many cells. For the purposes of this paper, we do
not require that the union of mesh cells corresponds to a simply
connected domain. We will rely on the very practical assumption that
the volume of each cell is positive and that cells are convex.

That said, the bulk of our arguments will not make use of this
geometric view of a triangulation. Rather,
it is convenient to reformulate the problem under consideration using
the language of graphs. When viewing a finite element mesh as an undirected    
graph, we consider it as a pair $G_\tria=(V,E)$ with the vertices 
$V=\{v_1,\ldots,v_{N_v}\}$ being the vertices of the mesh, and
edges $E=\{e_1,\ldots,e_{N_e}\}\subset \{\{v_a,v_b\}: v_a, v_b\in
V, v_a\neq v_b\}$  being the four ($d=2$) or twelve ($d=3$) edges of the
cells. Given the construction of this graph as a representation of a
mesh, there is then a collection of cells $\tria=\{K_1,\ldots,K_{N_K}\}$
where we can alternatively see each $K_i$ as either an ordered
collection of 4 vertices or an ordered collection of 4 edges (in 2d;
in 3d it is 8 vertices and 12 edges). The index $\tria$ on $G_\tria$
indicates that we are not considering general graphs but indeed only
those graphs that originate from a triangulation of a domain
$\Omega\subset\R^d$.

For a given cell $K$, let $v(K)\subseteq V$ be the set of its vertices
and $e(K)\subseteq E$ the set of its edges. For a given edge, let
$K(e)\subseteq{\tria}$ be the set of adjacent cells. In 2d, 
$|K(e)|$ is either one or two (depending on whether the edge is at the
boundary or not); in 3d, $|K(e)|\ge 1$ since arbitrarily many cells may be
adjacent to a single edge.

As discussed above, we are interested in assigning a direction to
each edge in such a way that the direction of the edge is implied from
the orientation of a cell. That is, for a mesh with associated graph
$G_\tria$, we would like to have a directed graph $G_\tria^+=(V,E^+)$ with the same vertex
set and edges, but where each edge is now considered directed (i.e.,
represented by an ordered pair of vertices). To be precise, this 
graph is in fact \textit{oriented} since we never have both $(v_i, v_j) \in E^+$ and 
$(v_j, v_i) \in E^+$ as may occur in directed graphs.

\subsection{Goals for orienting meshes}

As discussed in the introduction, practical implementations of the
finite element method need to define a coordinate system on both cells
and edges of a mesh. This is typically done by \textit{mapping} a
reference cell $\hat K=[0,1]^d$ and edge $\hat e=[0,1]$ to each cell
$K$ and edge $e$, along with the coordinate systems. The details of
this are of no importance here other than the following two
statements: (i) On each cell, we need to designate one of the four
(or eight) vertices as the ``origin''; each of the two (three) edges adjacent to the
origin then form the first and second (and third) ``coordinate
axis''. If we insist that the mapping from the reference cell $\hat K$
to $K$ has positive volume, then the choice of the origin fixes the
coordinate system in 2d; in 3d, each of the edges adjacent to the
origin can be chosen as the first coordinate axis, with the other two
then fixed. In other words, each cell allows for 4 possible choices of
coordinate systems in 2d, and $8\times 3=24$ in 3d. (ii) On each edge,
we define a coordinate system by choosing a ``first'' vertex.

With this, there are a total of $4^{|N_K|}$ choices in 2d for the coordinate systems
on the cells (and $24^{|N_K|}$ in 3d), and $2^{|N_e|}$ for the
coordinate systems of the edges. The question is now whether it is
possible to choose them in such a way that if the coordinate systems
of all cells are specified, we can infer the coordinate system on each
edge unambiguously, regardless of which cell adjacent to the edge we
consider. Or conversely: is it possible to specify directions for all
edges in such a way that
this implies a unique choice of coordinate systems for all cells?

For reasons that will become clear later, we will prescribe directions
of the edges of a cell when seen in the ``coordinate system'' of the
cell as shown in Fig.~\ref{fig:convention}. Clearly, it will be
possible to choose the coordinate systems of two neighboring cells in
such a way that they do not agree on the direction of the common
edge. Such a choice would require an implementation to store for each
cell whether or not an edge's (global) direction does or does not match the
direction that results from the (local) convention.

On the other hand, let us assume that we can find an orientation for
all edges of the mesh so that in each cell, ``opposite'' edges are
``parallel''. Then each cell has one vertex from which all oriented
edges originate, and we can choose this as the ``origin'' of the
cell. Using this choice of cell coordinate system, edge directions are
then again uniquely implied and the problem is solved. In other words,
we have reduced the problem of orienting all cells and edges to the
problem of finding one particular set of edge orientations that
satisfy the \textit{``opposite'' edges are
``parallel''} property.

This property is easy to understand intuitively. It is significantly
more cumbersome to describe it rigorously in the graph
theoretical language, and so we will only do this for the 2d case. (The
3d case follows the same approach but requires lengthy notation
despite the fact that the situation is relatively easy to understand.)
In order to formulate the convention, 
we need to fix the order in which we consider
vertices and edges as part of a cell. We do so using a lexicographic
order for vertices, as shown in
Fig.~\ref{fig:vertex-edge-convention}. Edges are numbered so that we
first number the edge from vertex 0 to 2, then its translation in
the perpendicular direction (i.e., from vertex 1 to 3), and then the
edges connecting the vertices of the first two edges. Both of these orders
reflect the tensor product structure of quadrilaterals and are easily
generalized to hexahedra (or higher dimensions, if desired). The
choice of edge directions within each cell, as shown in
Fig.~\ref{fig:convention}, then
ensures that the coordinate system of the edge is simply the
restriction of the cell's coordinate system to the edge.

\begin{figure}[tbp]
  \begin{center}
    \phantom{.}
    \hfill
    \includegraphics[width=.3\textwidth]{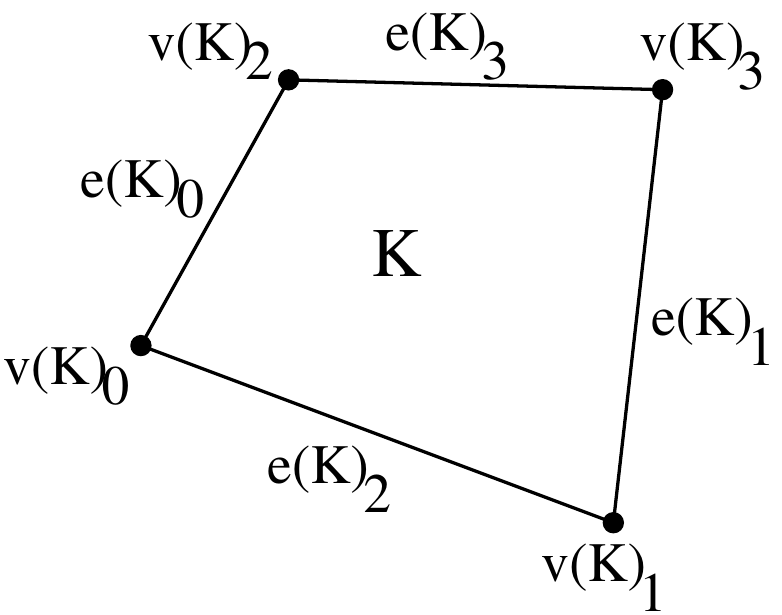}
    \hfill
    \phantom{.}
    \caption{\it Numbering convention for vertices and edges of
      two-dimensional cells. Here, $v(K)^+$ is the ordered set of
      vertices that bound the cell, and $e(K)^+$ the ordered set of
      edges.}
    \label{fig:vertex-edge-convention}
  \end{center}
\end{figure}

With this definition, each cell is described by an ordered tuple of
its vertices where we will assume that the first element of this tuple
corresponds to the ``origin''. Equivalently, we can describe each cell
as an ordered tuple of four (unordered) edges, where the ``origin'' of
the cell is now the common vertex of edges 0 and 2. Because there are
4 possible choices for the origin of the cell, there are four
ways to describe a cell that are equivalent up to rotation.

With these preparations, we can finally define what it means for the
edges of a directed graph $G_\tria^+$ that represents a quadrilateral mesh
to be consistently oriented:
\begin{convention}
  \label{conv:1}
  We call a graph $G_\tria^+=(V,E^+)$ \textit{consistently oriented with respect to cell $K\in\tria$}
  if among the four equivalent choices of vertex tuples of $K$, there is one so that the
  following directed edges are all elements of $E^+$:
  $e(K)_0=(v(K)_0,v(K)_2)$, 
  $e(K)_1=(v(K)_1,v(K)_3)$, 
  $e(K)_2=(v(K)_0,v(K)_1)$, 
  $e(K)_3=(v(K)_2,v(K)_3)$.
\end{convention}

\begin{convention}
  \label{conv:2}
  We call a graph $G_\tria^+=(V,E^+)$ \textit{consistently oriented} if it is
  consistently oriented with respect to all cells in $\tria$.
\end{convention}

As discussed above, a consistently oriented graph has edges that allow
us to choose a coordinate system on each cell so that the edge
orientations follow immediately from the cell orientations.
Similar definitions can be given for the 3d case. The purpose of this
paper is to ask the question whether it is always possible to
consistently orient the edges of a given mesh $\tria$, and if this is the case,
whether it can efficiently be done by an algorithm.

\subsection{Reformulation of Conventions \ref{conv:1} and \ref{conv:2}}

The developments
in the following sections all depend on the fact that
Convention~\ref{conv:1} can equivalently be stated by only looking at
sets of ``parallel edges'' of quadrilaterals:%
\footnote{The definition of whether edges are parallel given here only
  uses the graph theoretic context. In the language of finite element
  methods, it could equivalently be defined in geometric terms by
  using the coordinates of the vertices of the original mesh
  $\tria$. We can then view each edge of the
  graph as a (not necessarily straight) line connecting the adjacent
  vertices. Each cell $K\subseteq \tria$ occupies a
  subset of $\R^d$ that is the image of the reference square or cube
  $[0,1]^d$ under a homeomorphic mapping $\phi_K$. We can then
  call two edges $e_1,e_2$ parallel in $K$ if their preimages,
  $\phi_K^{-1}(e_1)$ and $\phi_K^{-1}(e_2)$, i.e., the corresponding
  edges of the reference cell, are parallel line segments in the geometric sense. This may be more
  intuitive, but we have no further use for mappings and
  transformations in this paper and will therefore not further explore
  the geometric setting.} 

\begin{definition}
  \label{def:parallel}
Two edges $e',e''\in E$ are called \textit{parallel edges of $K$} if $e'=e''$
or if they bound $K$ but do not share a vertex. If $e',e''\in E$ are parallel
edges of $K$, then we denote $e' \CellPar{K}
e''$. We say that $e',e''\in E$ are \textit{locally parallel} and denote $e'
\LocPar e''$ if there exists a cell $K \in {\cal K}$ so that $e' \CellPar{K} e''$.
\end{definition}
The four edges of a quadrilateral then fall into two classes of two edges each that are
parallel. (In 3d, a corresponding but more notationally cumbersome definition would yield three classes of
four parallel edges each.) In practice, one does not usually have to verify
if edges are parallel, but only enumerate classes of parallel edges;
this does not require testing all equivalent edge sets: in 2d, edges
$e(K)_0$ and $e(K)_1$ are parallel, as are $e(K)_2$ and $e(K)_3$, for
any arbitrarily chosen equivalent edge set.
We can then define consistent orientations via
these classes of parallel edges:

\begin{convention}
  \label{conv:3}
  Two oriented, parallel edges $e', e''\in E^+$ are called
  \textit{consistently oriented with respect to cell $K$} if
  $e'=(v(K)_0,v(K)_2)$, 
  $e''=(v(K)_1,v(K)_3)$ with regard to the one equivalent vertex set
  within which $e'=e(K)_0$ and $e''=e(K)_1$.
\end{convention}

\begin{convention}
  \label{conv:4}
  We call a graph $G_\tria^+=(V,E^+)$ \textit{consistently oriented with respect to cell $K$}
  if all sets of parallel edges of $K$ are consistently oriented
  with regard to $K$.
\end{convention}

In other words, consistent orientation on a cell can be tested by only
verifying consistent orientation of all the edges in all sets of
parallel edges. Consequently, \textit{testing} that a graph $G_\tria^+$ is
consistently oriented is relatively easy and can be done by verifying
the condition on every cell separately. A verification algorithm can
therefore easily be written with $O(N_K)$ complexity and,
consequently, with $O(N_e)$ because $\frac 12 N_K \le N_e \le 4N_K$.
On the other hand, \textit{generating}
a consistent orientation requires a \textit{global} algorithm because
the orientation of one edge implies that of all of its parallel edges
on all of its adjacent cells, which itself implies the orientation of
edges parallel on cells twice removed, etc. Because of this property,
it is not a priori clear that one can find a linear-time algorithm
that can find a consistently oriented graph $G_\tria^+$ given the graph $G_\tria$
associated with a mesh. However, as we will show below, this is indeed
possible.

As a final note in this section, let us state that for all algorithms that
follow, we assume that we have methods to generate the vertex and edge sets
$v(K),e(K)$ for a given cell $K$ with ${\cal O}(1)$ complexity. This can
easily be achieved by storing this information as the rows of $N_K\times 4$
matrices as is commonly done in all widely used software (in 3d, the vertex
adjacency matrix is of size $N_K\times 8$ and the edge adjacency
matrix is of size $N_K\times 12$). Furthermore, we will assume that finding
the cell neighbors of a given edge, $K(e)$, requires ${\cal O}(1)$ time. It is
obvious that in 2d, this can be achieved by storing an $N_e\times 2$ matrix
storing the indices of the one or two cells that are adjacent to each edge;
populating this matrix from $e(K)$ requires only a single loop over all
cells. In 3d, the number of cells adjacent to each cell can in principle be
equal to $N_K$, requiring data structures that can either not be queried in
${\cal O}(1)$ complexity, or can not be built with ${\cal O}(N_K)$ complexity;
however, in actual finite element practice, the meshes we consider will never
have more than, say, a dozen or so elements joined at any one edge, so that we
can consider $|K(e)|$ to be bounded by a constant, thereby allowing
storing information in tables of fixed width and ensuring that $K(e)$ can be queried
with ${\cal O}(1)$ complexity. These assumptions will be important to
guarantee the overall complexity of the algorithms we will consider in the
following.

As stated above, we assume that $K(e)$ can be \textit{evaluated} in ${\cal
  O}(1)$ time. To make this possible requires \textit{building} appropriate data structures, and
depending on what information is available this may require more than ${\cal
  O}(N)$ time. For example, if one only knew the vertices of each cell, i.e.,
$v(K)$, then building tables that can evaluate $e(K)$ in ${\cal O}(1)$ time
requires ${\cal O}(N\,\log N)$ time; this is the typical case with file
formats that store mesh information. On the other hand, if each cell already
knows all of its edge neighbors, as is often the case inside mesh generators
or finite element codes,
then building the tables to evaluate $e(K)$ in ${\cal O}(1)$ only costs ${\cal
O}(N)$ time and is consequently of the same complexity as the algorithms we
will discuss in the following.

\section{The two-dimensional case}
\label{sec:2d}

In this section, we will consider the two-dimensional problem, which can be
stated as follows: \textit{Given a graph $G_\tria$ that originates
  from a mesh $\tria$ composed of
quadrilaterals, find a consistently oriented graph $G_\tria^+$.}

As mentioned above, consistency of edge directions only requires us to ensure
consistency within sets of parallel edges. To this end, we will in the
following describe methods that first find all edges that are, in some sense
that goes beyond that defined in Definition~\ref{def:parallel}, parallel to
each other (Section~\ref{sec:2d-parallel}). We will then show how we can
consistently orient all edges that are in this sense parallel to each other
(Section~\ref{sec:2d-orienting-ribbon}) and finally how we can orient all
edges in the graph (Section~\ref{sec:2d-orienting-graph}).

\subsection{Decomposing $E$ into sets of parallel edges}
\label{sec:2d-parallel}

As we will show next, the set of edges $E$ of $G_\tria$ can be decomposed into
a collection of mutually exclusive sets of edges, where the edges in
each set are all parallel to each other in some global sense. 
This follows from the fact that the relation \emph{locally parallel} 
is reflexive (i.e., $e \LocPar e$ for all edges $e$) and symmetric (i.e. $e' \LocPar e''$ 
implies $e'' \LocPar e'$ for all edges $e',e''$) by construction. Let us define two 
edges $e'$, $e''$ to be globally parallel if there is a finite, possibly empty sequence of 
edges $e_0, e_1, \ldots, e_s$ such that $e' \LocPar e_0 \LocPar e_1 \LocPar \ldots \LocPar 
e_s \LocPar e''$ and denote this by $e' \GlobPar e''$. One immediately checks that this 
relation is reflexive, symmetric and transitive (i.e. $e' \GlobPar e''$ and $e'' \GlobPar 
e'''$ imply $e' \GlobPar e'''$ for all edges $e',e'',e'''$) and hence forms an equivalence 
relation.%
\footnote{In an abstract sense, we have constructed the relation $\GlobPar$ as 
the \emph{transitive hull} of the relation $\LocPar$.}
This relation then partitions $E$ into disjoint sets of (globally) parallel edges 
\cite[Chapter I, \S 3]{Kur63} or \cite[Corollary 28.19]{LP98}.

Algorithmically, we can
recursively construct each of these sets by starting with an edge $e$ and build
the set $\parallelset(e)\subseteq E$ of all edges
that are globally parallel to $e$ by first
adding the edges opposite $e$ in the cells adjacent to $e$, then those
edges that are opposite to the ones added previously, and so on.  Sets
of parallel edges are central to the rest of the paper, since our edge
direction convention requires that they will all have parallel
directions.  An intuitive interpretation of the importance of parallel edges is as
follows: assume we had already found the directed graph $G_\tria^+$. Then,
flipping the direction of an edge $e$ would make the triangulation
non-consistent, and to make it consistent again we would have to flip
the directions of a number of other edges as well; the entire set of
edges that needs to be flipped, including $e$, is precisely $\Pi(e)$.

With this knowledge, let us concisely define an algorithm to find all
elements of a parallel set for an edge $e$ as a first step:

\begin{algorithm}[(Construct one set of parallel edges)]
  \label{alg:parallel-edges}
  Let $e\in E$ be a given edge and generate the set $\parallelset(e)$ of edges
  parallel to $e$ recursively as follows, where the set $\Delta_k$ consists
  of those edges that we add to $\Pi(e)$ in each step as we grow it away from $e$:
  \begin{enumerate}
  \item Set $\parallelset_0(e)=\emptyset$, $\parallelset_1(e)=\{e\}$,
    $\Delta_0=\{e\}$, $k=1$.
  \item While $\Delta_{k-1}\neq \emptyset$, do:
    \begin{enumerate}
    \item Set $\Delta_k=\emptyset$.
    \item For each $\delta\in\Delta_{k-1}$:
      \begin{itemize}
      \item Set ${\cal E}(\delta)=\emptyset$.
      \item For each $K\in K(\delta)$:
        \begin{enumerate}
        \item Set ${\cal N}_K(\delta):=\{\varepsilon \in e(K):
          \varepsilon \CellPar{K} \delta \} \backslash \{\delta\}$.
        \item Set ${\cal E}(\delta) = {\cal E}(\delta) \cup {\cal N}_K(\delta)$.
        \end{enumerate}
      \item Set $\Delta_k = \Delta_k \cup \left({\cal E}(\delta) \backslash\parallelset_k(e))\right)$.
      \end{itemize}
    \item $\parallelset_{k+1}(e) = \parallelset_k(e) \cup \Delta_k$.
    \item Set $k:=k+1$.
    \end{enumerate}
    \item Set $\parallelset(e)=\parallelset_k(e)$.
  \end{enumerate}
\end{algorithm}

Fig.~\ref{fig:parallel-sets} gives a graphical depiction of the construction
of one such set. Starting from a given edge $e$, the set
$\parallelset(e)$ grows in each step by
at most the two edges in $\Delta_k$ along a line that always intersects cells from
one side to the opposite one, and connects these parallel edges.
Growth of the set in one direction stops whenever this line hits a boundary
edge (in which case the set of opposite edges for this boundary edge $\delta$,
${\cal E}_K(\delta)$, has only one member, which furthermore is already in
$\parallelset_{k-1}(e)$), or if both ends of the line meet ``head-on'' (in which case all
elements of ${\cal E}(\delta)$ for all $\delta\in\Delta_{k-1}$ are already in
$\parallelset_{k-1}(e)$ and thus $\Delta_k=\emptyset$, upon which the
iteration terminates).

\begin{figure}[tbp]
  \begin{center}
    \includegraphics[width=.2\textwidth]{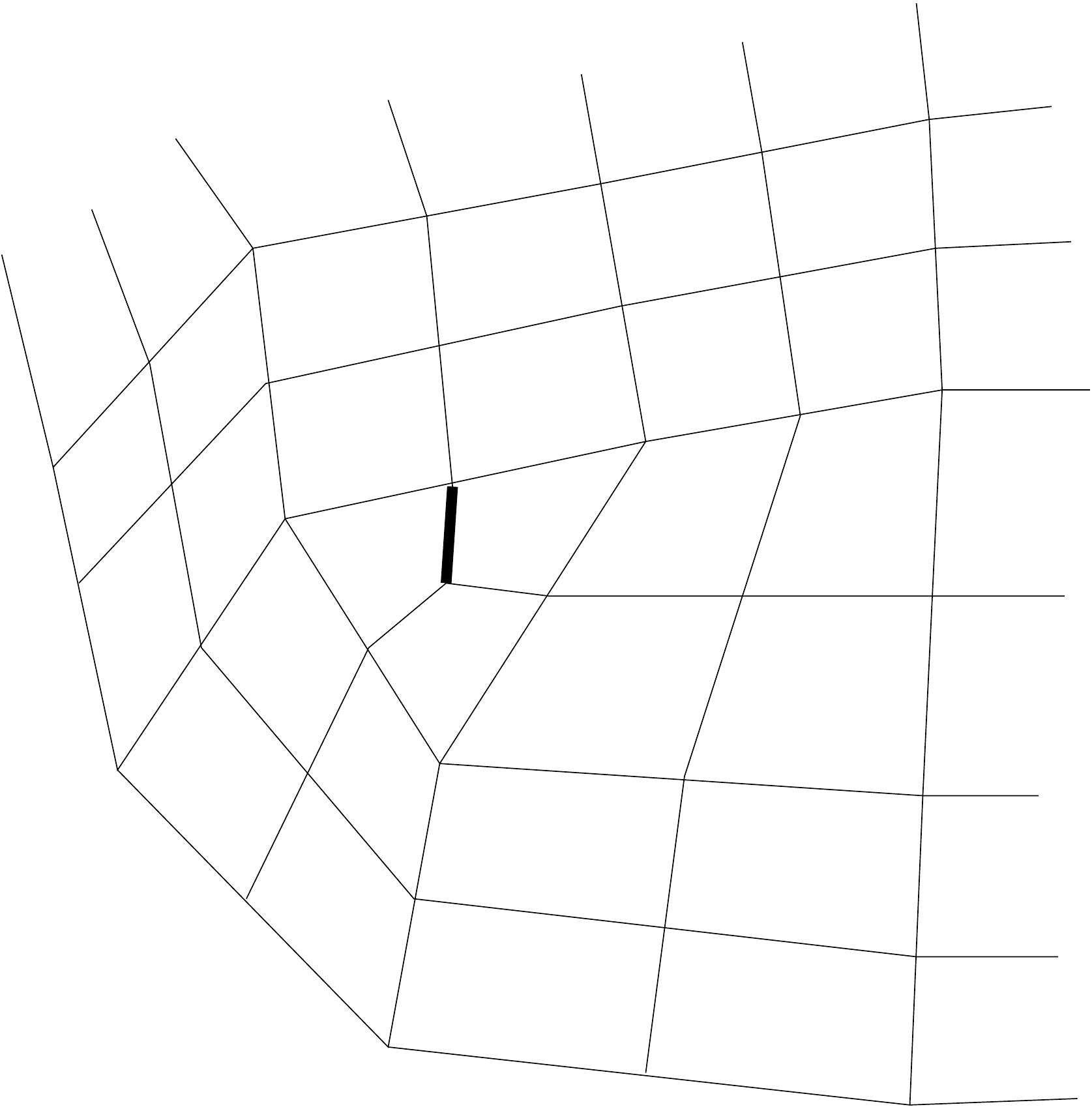}
    \hfill
    \includegraphics[width=.2\textwidth]{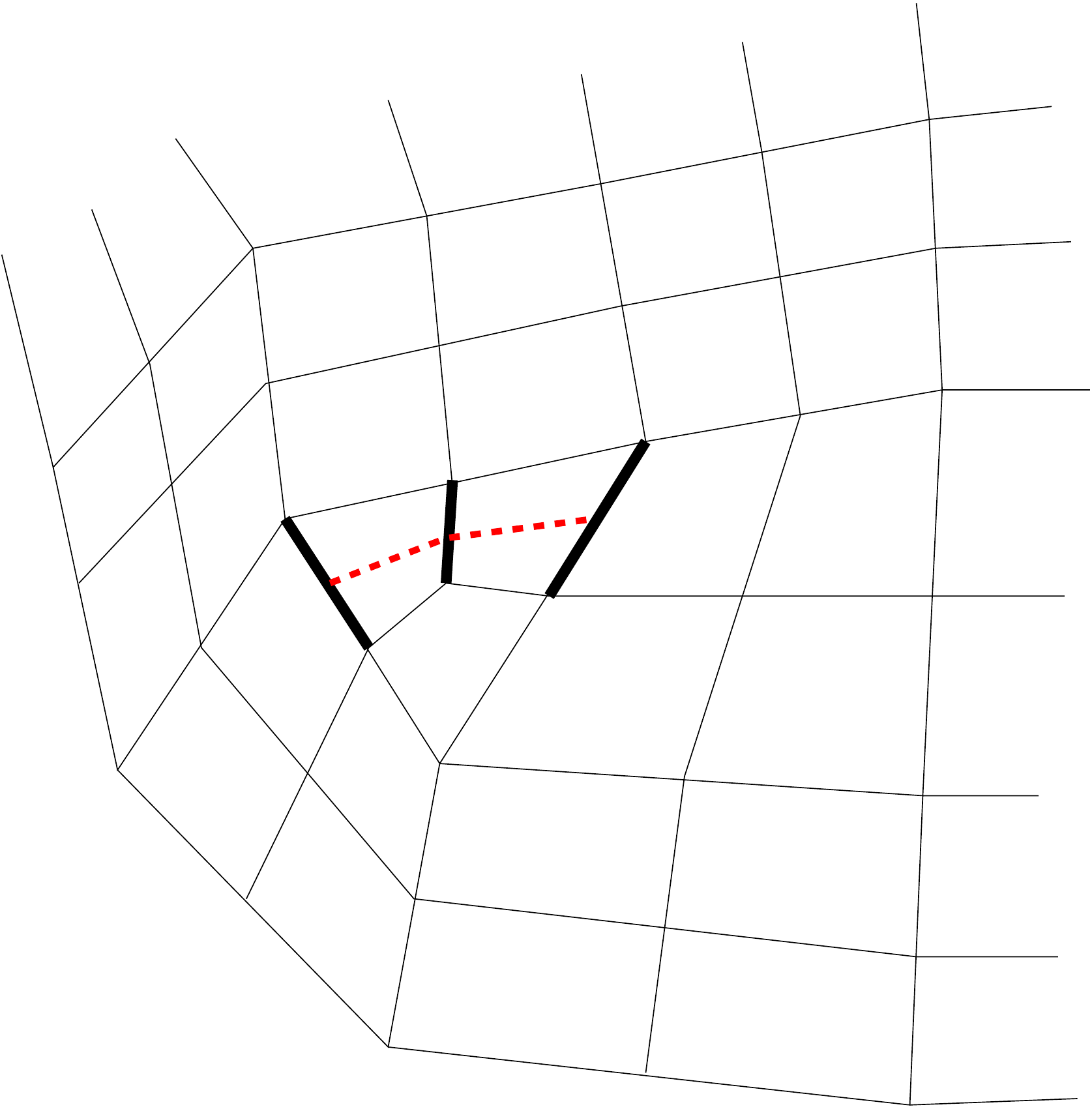}
    \hfill
    \includegraphics[width=.2\textwidth]{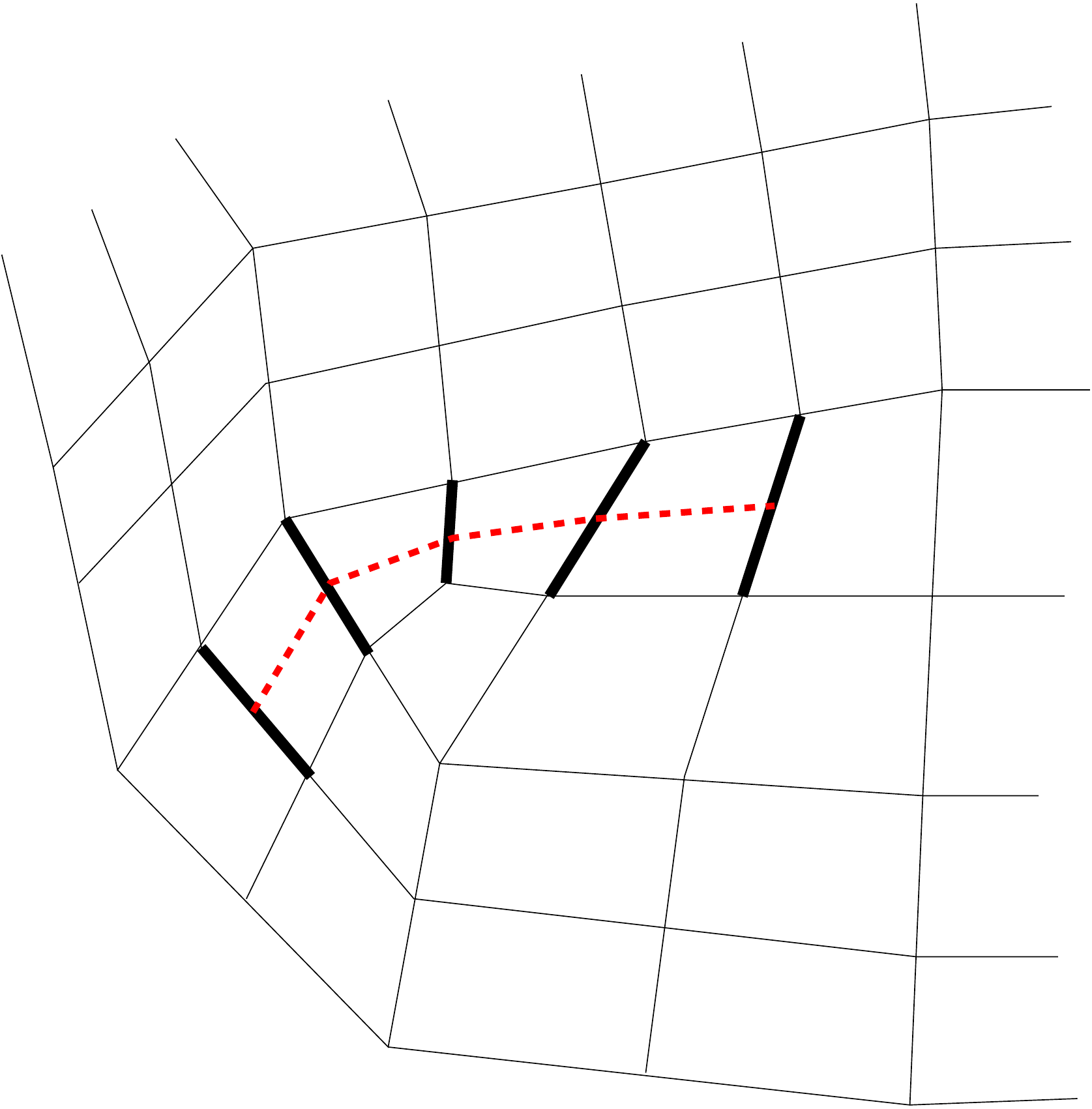}
    \hfill
    \includegraphics[width=.2\textwidth]{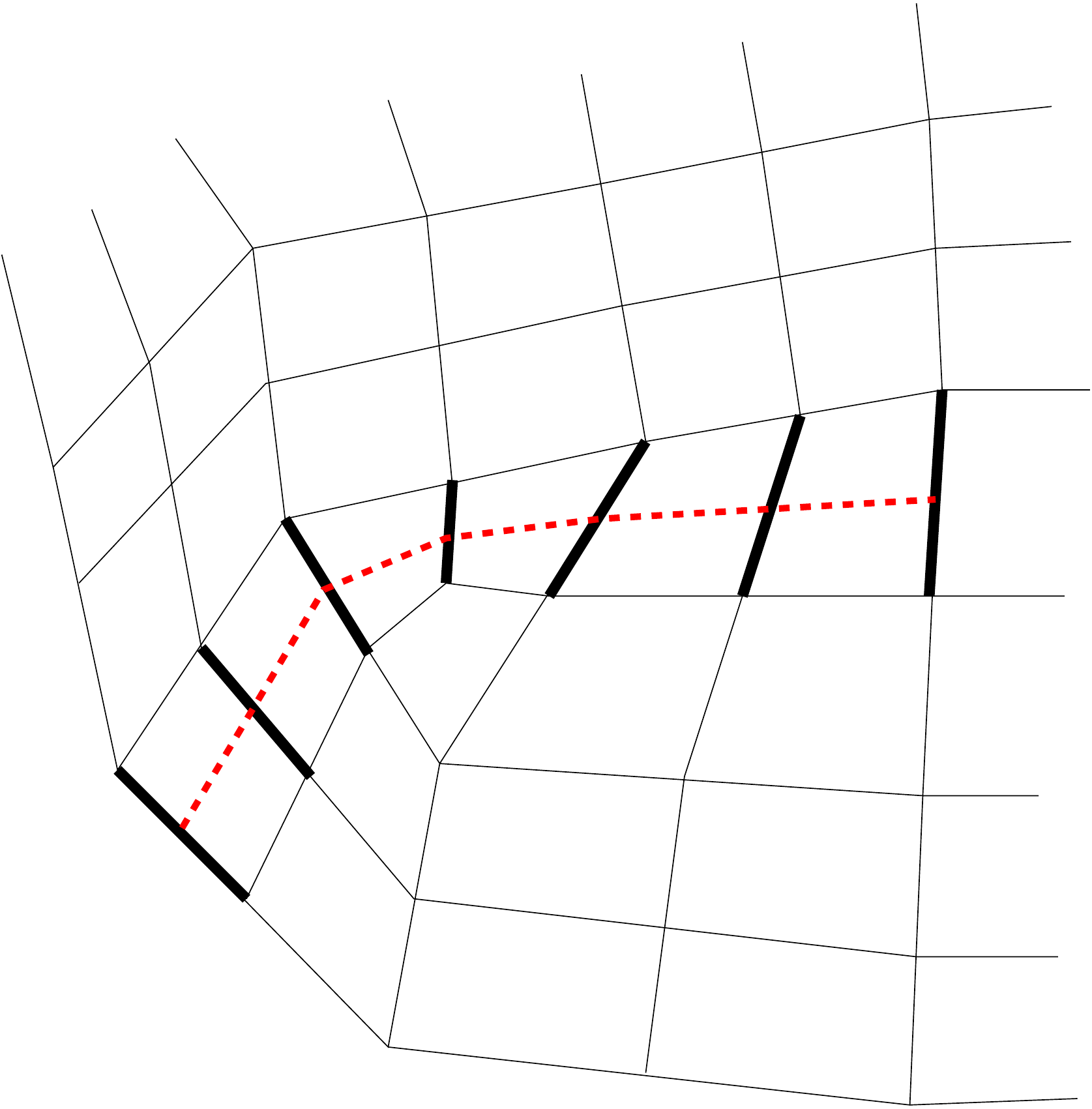}
    \caption{\it Construction of the parallel sets $\parallelset_k(e)$,
    $k=1,\ldots, 4$, indicated by bold edges. The red line connects all of
    them and grows in both directions.}
    \label{fig:parallel-sets}
  \end{center}
\end{figure}

To assess the overall run time of this algorithm, we note that in 2d, each
edge has exactly two neighboring cells unless it is at the boundary,
i.e., $|K(\delta)|\le 2$. Furthermore, within each cell, there is
exactly one other parallel edge to a given edge $\delta$, i.e.,
$|{\cal N}_K(\delta)|=1$. Consequently, in step (2)(b) we have 
$|{\cal E}(\delta)|\le 2$ and it follows that $|\Delta_k|\le 2$. With the
appropriate data structures -- for example, by representing sets of
known maximal cardinality through fixed-sized arrays --, all
operations in steps (2)(a)--(d) can then be executed in 
${\cal O}(1)$ operations. Furthermore, since $\parallelset_k(e)$ grows by
one or two elements per iteration, the loop represented by step (2)
executes at most $|\parallelset(e)|$ times. The total cost of the
algorithm is therefore ${\cal O}(|\parallelset(e)|)$.

The next step is based on the realization that every edge $e$ in the
graph $G_\tria$ can be uniquely sorted into one of a collection of mutually
exclusive sets
$\pi=\{\parallelset_1,\ldots,\parallelset_{N_\parallelset}\}$.
Each class $\Pi_i$ is constructed as above. Because the connecting
line for each parallel set is either closed or ends on both sides at
the boundary, the number of distinct sets of parallel edges, $|\pi|$, is at least
half the number of boundary edges, but of course at most half the number of edges
in $G_\tria$.

Algorithmically, we can construct the collection $\pi$ of parallel sets in the following way:
\begin{algorithm}[(Construct the set of parallel edge sets)]
  \label{alg:all-parallel-edges}
  Construct the set of parallel edge sets as follows:
  \begin{enumerate}
  \item Set $\pi=\emptyset, {\cal E}=E$.
  \item While ${\cal E} \neq \emptyset$, do:
    \begin{enumerate}
    \item Choose any $e \in {\cal E}$.
    \item Compute $\parallelset=\parallelset(e)$ using Algorithm~\ref{alg:parallel-edges}.
    \item Set $\pi = \pi \cup \{\parallelset\}$.
    \item Set ${\cal E} = {\cal E} \backslash \parallelset$.
    \end{enumerate}
  \end{enumerate}
\end{algorithm}

Here, the set of not-yet-classified edges ${\cal E}$ is reduced one
equivalence class -- i.e., by one set of globally parallel edges -- at a
time. Because the decomposition of edges into
equivalence classes is unique, it is clear that in each iteration,
$\parallelset\subseteq{\cal E}$. Furthermore,
$\parallelset(e)\supseteq\{e\}$ and so $|\parallelset|\ge 1$; thus,
the iteration is guaranteed to terminate.

More concretely, the cost of each iteration is given by
Algorithm~\ref{alg:parallel-edges}, i.e., ${\cal O}(|\Pi(e)|)$. The
overall cost is therefore $\sum_{\parallelset\in\pi} {\cal
  O}(|\Pi|)$. On the other hand, because edges can be uniquely sorted
into their equivalence classes, we know that 
$\bigcup_{\parallelset\in\pi} \parallelset = E$. Thus, the cost of
Algorithm~\ref{alg:all-parallel-edges} is ${\cal O}(|E|)$, i.e., of
optimal complexity.

\subsection{Orienting the elements of a set of parallel edges}
\label{sec:2d-orienting-ribbon}
Our convention was only that opposite edges in a cell have parallel
directions, but there was no requirement on the relative directions of
adjacent (non-opposite) edges within a cell. In fact, that was the
basis for restating Conventions~\ref{conv:1} and \ref{conv:2} in terms of
Convention~\ref{conv:3} and \ref{conv:4}. It is thus easy to see that we
only have to make sure that we have consistent directions of all edges within
each set of parallel edges, and that consistency of edges within each such set
is independent of the directions of edges in all other parallel sets. The
following lemma proves that within each such equivalence class a
consistent set of directions exists:

\begin{lemma}
  \label{lemma:existence-for-one-parallel-set}
  Let $e\in E$. Then there exists a choice of orientations for the elements of
  $\parallelset(e)$ that is consistent; i.e., for all
  $e',e''\in\parallelset(e)$ so that $e' \CellPar{K} e''$ for some
  cell $K$, then the
  orientations we associate with $e'$ and $e''$ are consistent in $K$.
\end{lemma}

This statement can be proven in a variety of ways, both constructively
and in indirect ways. The most intuitive way uses the
fact that a curve in the plane, closed or not and possibly
self-intersecting, such as the dotted line in
Fig.~\ref{fig:parallel-sets}, allows for the definition of a unique direction
``from one side of the curve to the other'', and we can then orient
each edge it crosses according to this direction. This statement appears
obvious on its face. For closed, non-intersecting curves, it
follows from the Jordan Curve Theorem that states that such curves
partition the plane into an ``inside'' and ``outside'' area and we can
then, for example, choose the direction from the inside to the outside to orient
edges. Proving the existence of such a direction field for
self-intersecting curves requires more work. The history of the Jordan
Curve Theorem teaches us that care is necessary, and our variations
on a proof typically required a page or more of geometry, even when
taking into account that we only need to show the statement for the
piecewise linear curves that connect edge midpoints.

Thus, rather than providing such a proof here, let us for now
simply consider the lemma to be true. A proof will later follow from
Remark~\ref{remark:extension-manifold} and a statement in the appendix
where we show a more general statement of which
Lemma~\ref{lemma:existence-for-one-parallel-set} is simply a special case.

Given this statement of feasibility, we can now ask for an
algorithm that assigns directions to all edges in a set
$\parallelset(e)$ and that can be implemented with ${\cal
  O}(|\parallelset(e)|)$ complexity. We do so by extending
Algorithm~\ref{alg:parallel-edges} to include edge orientation
assignment:

\begin{algorithm}[(Orient edges for consistency)]
  \label{alg:orient-edges}
  Let $\pi=\{\parallelset_1,\ldots,\parallelset_{N_\parallelset}\}$ be
  given. Then perform the following operations for each $i=1,\ldots,N_\parallelset$:
    \begin{enumerate}
    \item Assign an ``unassigned'' orientation to each edge $e\in \parallelset_i$.
    \item Choose some $e\in \parallelset_i$ and set $\Delta_0=\{e\}, k=1$.
    \item Assign an arbitrary orientation to $e$.
    \item Set $\Delta_k=\emptyset$.
    \item While $\Delta_{k-1}\neq \emptyset$, do:
      \begin{enumerate}
      \item For each $\delta\in\Delta_{k-1}$:
        \begin{itemize}
        \item Set ${\cal E}(\delta)=\emptyset$.
        \item For each $K\in K(\delta)$:
          \begin{enumerate}
          \item Set ${\cal N}_K(\delta):=\{\varepsilon \in e(K):
            \varepsilon \CellPar{K} \delta \} \backslash \{\delta\}$.
          \item Assign an orientation to the elements of ${\cal N}_K(\delta)$
            that is consistent in $K$ with that of $\delta$.
          \item Set ${\cal E}(\delta) = {\cal E}(\delta) \cup {\cal N}_K(\delta)$.
          \end{enumerate}
        \item Set $\Delta_k = \Delta_k \cup \left({\cal E}(\delta) \backslash\parallelset_k(e))\right)$.
        \end{itemize}
      \item $\parallelset_{k+1}(e) = \parallelset_k(e) \cup \Delta_k$.
      \item Set $k:=k+1$.
      \end{enumerate}
    \end{enumerate}
\end{algorithm}

It is important to note that in step (5)(a)(ii), we assign directions only to
the cell neighbors ${\cal N}_K(\delta)$ of $\delta$, all of which either do
not yet have an orientation, or that have already been assigned
that very same orientation when coming ``from the other side of the parallel set''
in case the connecting line for this set is a closed line through our
mesh. The step therefore never changes an already assigned orientation.

The algorithm above would, in practice, simply be implemented as part of
computing the parallel sets $\parallelset_i$. In that case, it isn't even
necessary to explicitly build $\parallelset_k(e)$ in step (5)(b) as this set
is only used in the last part of step (5)(a) where we could simply exclude all
elements from ${\cal E}(\delta)$ that had previously already been assigned an
orientation.

\subsection{Orienting all edges in a graph}
\label{sec:2d-orienting-graph}
With these results, we can state the main result for the case $d=2$:

\begin{theorem}
  \label{theorem:2d}
  For every planar, undirected graph $G_\tria$ generated by subdividing a bounded
  domain in $\R^2$ into finitely many quadrilaterals, there exists a
  corresponding, consistently oriented directed graph $G_\tria^+$.
\end{theorem}
\begin{proof}
  The proof follows from the preceding subsections: first, we can uniquely sort
  all edges into equivalence classes, for which we can choose edge directions
  independently; second, we can find a consistent choice of edge
  directions within each such set.
\end{proof}

Since both sorting edges into equivalence classes and assigning
directions to edges of all equivalence classes are linear in
the number of edges in the equivalence set, the overall algorithm is
${\cal O}(|E|)$ and therefore order optimal in the number of edges.
Furthermore, because each edge is shared by no more than two cells and
because each cell is bounded by exactly four edges, it is obvious that
$\frac 14 |E| \le |{\tria}| \le 2|E|$. It
follows directly that the algorithm is not only linear in the number of
edges, but also in the number of cells (which is the more important estimate
in practice).%
\footnote{In practice, not only the complexity class matters, but also
  the constant in front of it. Our reference implementation in
  \dealii{} orients the edges of the roughly 30,000 cells of the mesh
  shown in the left panel of
  Fig.~\ref{fig:wing} in 0.035 seconds on a current laptop -- far
  faster than solving any equation on this mesh would take.}

\begin{remark}
  \label{remark:extension-manifold}
  The arguments above showing that 2d meshes are always orientable can
  be carried over to meshes on two-dimensional, orientable surfaces,
  and we will prove so in the appendix. In particular, this
  holds for the practically important case of 
  2d meshes covering (part of) the surface of a 3d domain. Indeed,
  this is true whether the domain is homeomorphic to the
  unit ball or not -- i.e., meshes on the surface of 3d domains with
  handles are still orientable.

  The proof of this more general statement then also covers
  Lemma~\ref{lemma:existence-for-one-parallel-set}: the lemma states
  the orientability of edges of a mesh in the plane $\R^2$, which is
  obviously an orientable, two-dimensional manifold. The proof of
  Lemma~\ref{lemma:existence-for-one-parallel-set} thus follows from
  the proof given in the appendix.

  On the other hand, meshes on non-orientable surfaces,
  for example on a M{\"o}bius strip, are not necessarily
  orientable. This is because some of the connecting lines of parallel
  edges (see Fig.~\ref{fig:parallel-sets}) may wrap around the strip
  and return what we thought to be a vector from ``one side to the
  other side'' in its reverse orientation. It is
  therefore not possible to define a unique ``right'' and ``left'' of
  a curve on a non-orientable manifold, and not all sets of edges of a
  mesh defined on it may have a consistent orientation.
\end{remark}

\begin{remark}
  \label{rem:circular-edges}
  As mentioned in a footnote to Section~\ref{sec:introduction} and
  Fig.~\ref{fig:convention}, one can imagine other conventions for the
  relative orientations of edges and cell. For example, for triangles one
  often assumes a circular orientation. Such
  a convention could also be adopted for quadrilaterals. However, it has two
  problems: (i)~It is not obvious how to generalize it to
  hexahedra. (ii) It is easy to construct meshes for which no set of
  globally consistent edge orientations can be found. To see this, note that a
  circular choice for edge directions in each cell implies that opposite edges
  have anti-parallel directions. Then imagine a closed string of $n$
  cells. One of the sets of parallel edges then contains all of the $n$ edges
  that separate the $n$ cells forming a circle. The convention requires
  us to orient them alternatingly -- something that is only possible
  if $n$ happens to be even.
\end{remark}

\section{The three-dimensional case}
\label{sec:3d}

Let us now also look at the case of three spatial dimensions, and a
subdivision of the domain into hexahedral cells. The right panel of
Fig.~\ref{fig:convention} shows the convention for $d=3$, where again we will
want to have all parallel edges point in the same direction. Note that now we
have three sets of four edges within each cell.

We will have to investigate both whether the algorithms discussed in
the previous section always yield a consistent edge orientation, and
what their complexity is given the changed circumstances.

\subsection{Are the edges of 3d meshes always orientable?}
\label{sec:3d-always}

The first step is to construct the sets of parallel edges,
$\pi$. Algorithm~\ref{alg:parallel-edges} again finds all
elements of an equivalence set $\parallelset(e)$ for a given starting
edge $e$. Instead of following a line intersecting
opposite edges starting at $e$ in both directions, we now have to
follow a sheet going through all four parallel edges of a hexahedron.
It is easy to see that again there is a unique classification of all
edges into equivalence sets of parallel edges.

The second step was to assign a consistent orientation to the edges in a set
of parallel edges. This was possible for $d=2$ since the line that connects
all these edges is always orientable. However, that is not the case for the
sheet connecting the edges in $d=3$: It may not be orientable, and in this
case we will not be able to find a consistent orientation for the edges in this
equivalence class because we can no longer choose their directions to
be from ``one side'' to the ``other side'' of the sheet. The following
two sections show examples of meshes whose edges are not orientable according
to our conventions.

\subsubsection*{A first counterexample}

\begin{figure}[tbp]
  \begin{center}
    \includegraphics[width=.65\textwidth]{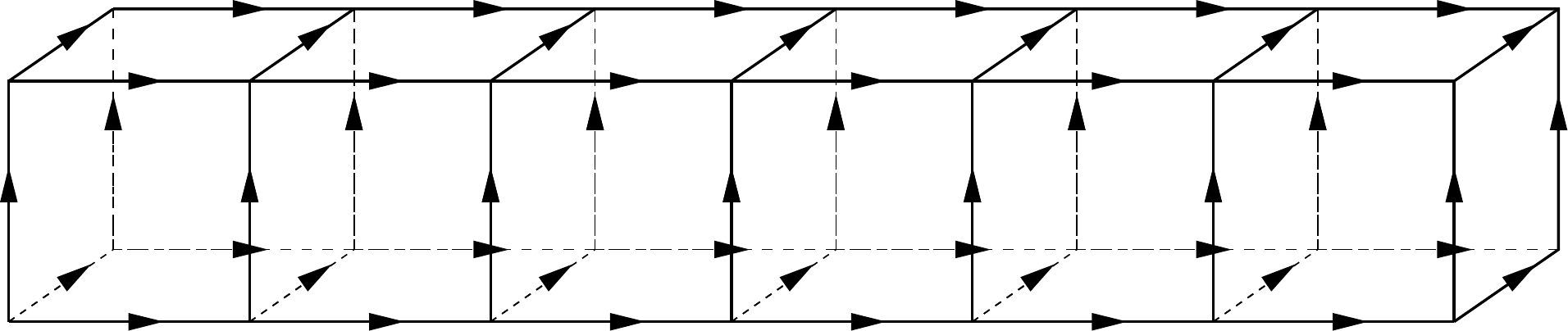}
    \\[6pt]
    \phantom{.}
    \hfill
    \includegraphics[width=.4\textwidth]{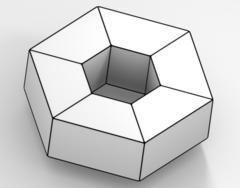}
    \hfill
    \includegraphics[width=.4\textwidth]{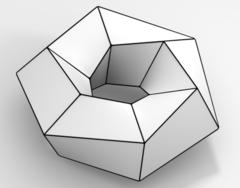}
    \hfill
    \phantom{.}
    \\[6pt]
    \phantom{.}
    \hfill
    \includegraphics[width=0.4\textwidth]{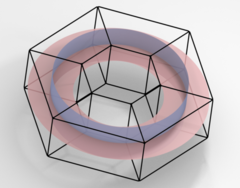}
    \hfill
    \includegraphics[width=0.4\textwidth]{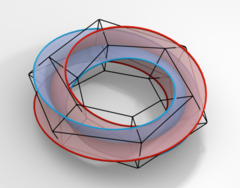}
    \hfill
    \phantom{.}
    \caption{\it First counterexample for $d=3$. Taking a string of cells
      (top) and bending them into a torus yields a graph for which a
      consistent orientation of edges exists (center left). However, twisting
      the cells by $180^\circ$ before re-attaching ends fails to yield such a
      graph (center right): The sheets connecting the radial and axial
      parallel edges form two intersecting,
      non-orientable M\"obius strips (bottom).}
    \label{fig:torus}
  \end{center}
\end{figure}

A first non-orientable example is shown in Fig.~\ref{fig:torus},
demonstrating a subdivision of a toroidal domain for which no
consistent edge
orientation exists. If we take a string of cells (top) and bend it into a
torus (bottom left), then all edges can be grouped into one radial, one axial,
and $|{\tria}|$ tangential classes (the surfaces connecting the edges of
the first two classes are shown in the bottom row of Fig.~\ref{fig:torus}).
One possible consistent orientation would be radially inward and axially into
positive $z$-direction. The tangential edges can be oriented arbitrarily for
each cell separately. However, such a consistent choice of direction for edges
no longer exists if we twist the string of cells by $180^\circ$ before
attaching ends (bottom right). In that case, there must be at least one cell
with radial edges that both point inward and outward, in violation of our
convention. The same holds for a twist of $90^\circ$, in which case the sheet
connecting parallel edges has to circle the torus twice, before meeting itself
in the wrong orientation again. The sheet that passes through parallel edges
is, in these cases, a M\"obius strip with either a twist of $180^\circ$ or
$90^\circ$, and it is well known that this surface is not orientable.

\subsubsection*{A second counterexample}

Conjecturing from the first example that triangulations into hexahedra with no
consistent orientation must have a hole or be multiply connected, is wrong,
though. Fig.~\ref{fig:rainers-example} shows an example of 14 hexahedra
subdividing a simply connected domain without holes for which no consistent
orientation exists. The bottom row of the figure shows the top face of the seven
lower hexahedra (left) and the bottom face of the upper seven (right). Only
faces $A$ and $B$ match and are connected.

\begin{figure}[tbp]
    \begin{center}
    \phantom{.}
    \hfill
    \includegraphics[width=.44\textwidth]{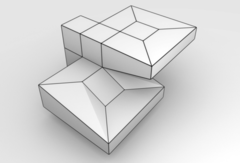}
    \hfill
    \includegraphics[width=.44\textwidth]{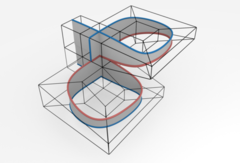}
    \hfill
    \phantom{.}
    \\
    \phantom{.}
    \hfill
    \includegraphics[height=.17\textwidth]{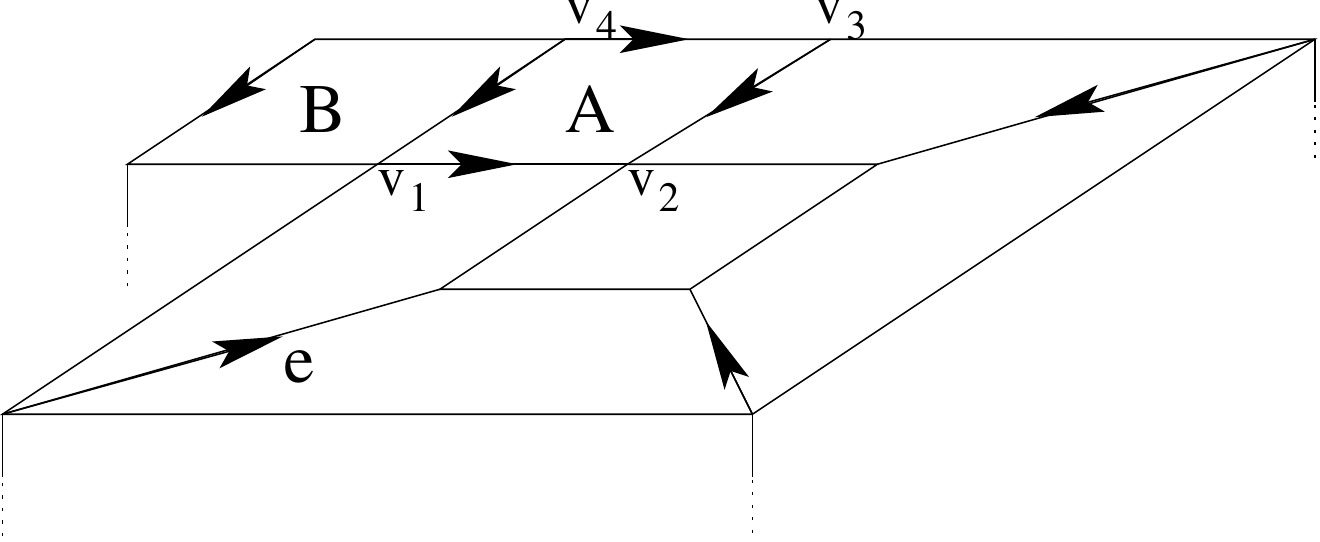}
    \hfill
    \includegraphics[height=.17\textwidth]{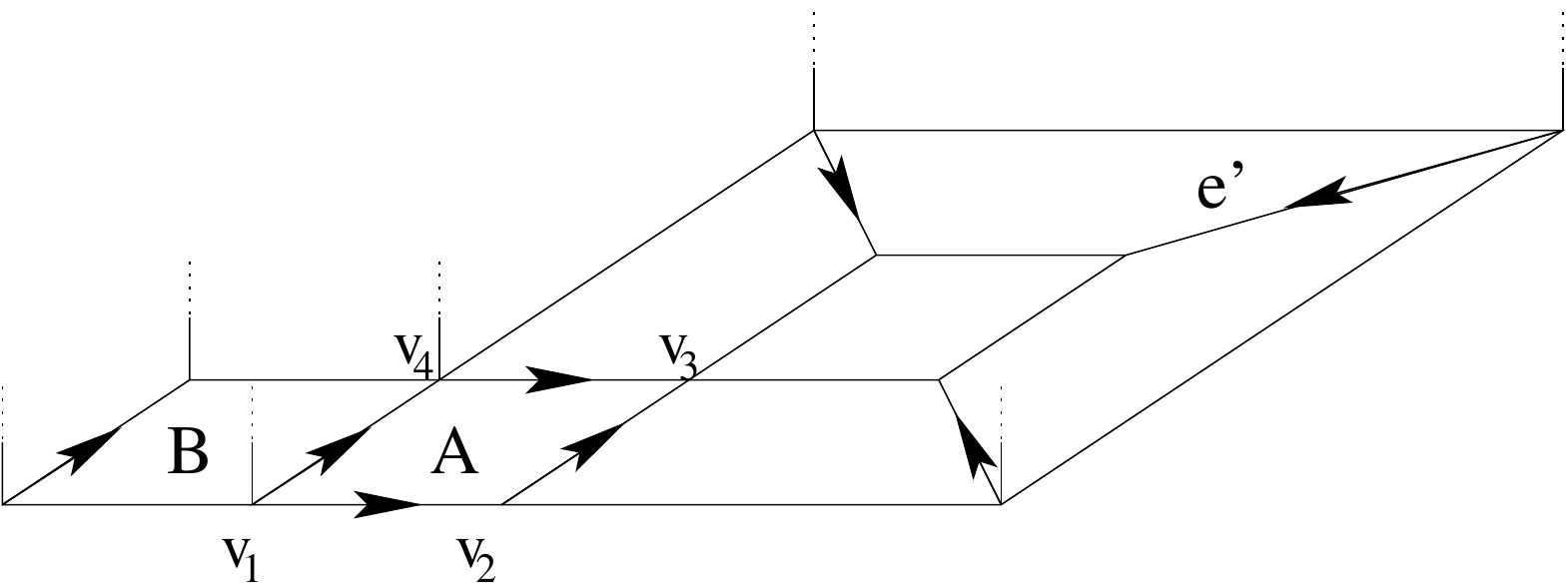}
    \hfill
    \phantom{.}
    \caption{\it Second counterexample for $d=3$ with a simply connected
    domain. Top row: Hexahedralization and one of the sheets passing through
    parallel edges. Bottom row: Top face of lower half of domain and
    bottom face of upper part of domain, both showing consistent directions of
    edges if one oriented the edge $(v_4,v_3)$ as shown. This leads
    to conflicting directions for $(v_4,v_1)$, and the same happens
    if one had oriented $(v_4,v_3)$ in the opposite direction. Note that the
    two parts shown are connected
    only on faces $A$ and $B$.}
    \label{fig:rainers-example}
  \end{center}
\end{figure}

Let us consider the orientation of the radial edges of the lower left
picture: starting, for example, at edge $e$, then all radial edges
must either point inward or outward due to the opposite edge
rule. Also, the direction of the other two edges of face $A$ is then
fixed. One of the two possibilities for directions of these edges is shown in
the bottom left part of Fig.~\ref{fig:rainers-example}. Independent of these
two possible choices, we have the \textit{invariant} that $v_2$ and
$v_4$ are vertices to which both lines in $A$ either converge, or from
which they emanate, and $v_1$ and $v_3$ are vertices into which one
line enters and from which one line emerges.

Now let us consider the underside of the upper seven hexahedra, given by
the lower right figure. It is connected to the lower part along faces
$A$ and $B$. By the same argument, starting for example at
edge $e'$, the directions of the radial edges and the other two edges
of face $A$ are fixed. However, this time vertices $v_1$ and $v_3$
are the vertices from which both adjacent edges in $A$ emerge and into
which they vanish, and $v_2$ and $v_4$ are the ones through which
they pass, i.e.~the character of the vertices is changed. Since in the
joint domain vertices and edges of the upper- and undersides are
identified, this creates a conflict: whatever direction we choose for
edges $e$ and $e'$, no consistent edge orientation of the face $A$
is possible.

This is easily understood by looking at the sheet that connects the
set of parallel edges, $\parallelset(e)=\parallelset(e')$,
shown in the upper right part of the figure. It is a split,
self-intersecting sheet that is not orientable.

\subsection{Algorithm complexity}

The examples in the previous subsection show that not all 3d meshes
allow for a consistent edge orientation. In a sense, the complexity of
our algorithm may therefore be a moot point: codes do necessarily have
to store explicit edge orientation flags for each cell because the
orientation of edges can no longer be inferred from the orientation of
the cell.

On the other hand, it may still be of interest to orient edges
consistently for those meshes for which this is possible, for example
to ensure that the code paths for default edge orientation are always
chosen. We may also be interested in seeing whether our algorithm can
at least detect in optimal complexity whether a mesh is orientable. 

To investigate these questions, we note that
Algorithms~\ref{alg:parallel-edges} and \ref{alg:all-parallel-edges}
separating edges into the set $\pi$ of equivalence classes continue to
work. Algorithm~\ref{alg:parallel-edges} determines the overall
complexity. In 3d, we need to note that in step (2)(b)(i), ${\cal
  N}_K(\delta)$, the set of edges parallel to $\delta$ in cell $K$,
now has cardinality 3. Furthermore, the loop in step (2)(b) now
iterates over all cell neighbors $K(\delta)$, of which in 3d there may
in fact be up to $|{\tria}|$. Consequently, we can no longer bound
$|{\cal E}(\delta)|$ by a constant independent of the number of edges
or cells, and this then applies to the cost of the entire step
(2)(b). This would not be a problem if we added at least a fixed
fraction of the elements of ${\cal E}(\delta)$ to $\Delta_k$ (i.e., if
we could guarantee that $\frac{|{\cal E}(\delta)\backslash
  \parallelset_k(e)|}{|{\cal E}(\delta)|} \ge c >0$), but we are not
aware of any theoretical argument that this would be so.
Consequently, it is conceivable that there exist sequences of meshes
for which $|{\cal E}(\delta)|={\cal
  O}(|E|)$ but for which step (2)(c) adds only a fixed number of
elements to $\parallelset_k$ (or at least a number that grows less
quickly than $|E|$). This would destroy the linear complexity
of the overall algorithm.

From a practical perspective, however, this question is not terribly
interesting as it requires meshes in which the number of cells
adjacent to individual edges becomes very large. While the optimal
number of cells adjacent to each edge would be four (for example in a
cubic lattice), practical meshes generated by mesh generators rarely
have more than 8 or 10 such adjacent cells per edge, and this number
is independent of the number of cells. Consequently, for
such meshes, $|{\cal E}(\delta)|={\cal O}(1)$ and the algorithm is
then again guaranteed to run with optimal complexity. We believe that
it would also be possible to reformulate the algorithms in ways that
allow for optimal complexity even for these pathological cases, though
this is of no interest in the current paper.

Finally, we note that when a mesh is not orientable, step (5)(a)(ii) of
Algorithm~\ref{alg:orient-edges} will eventually try to assign a direction to
an edge that already has an orientation that is inconsistent with the one that
we want to assign to it (a case that cannot happen in 2d). Thus, the algorithm
will be able to detect non-orientable meshes with the same run time complexity
as it can orient meshes.

\section{Meshes in $d=3$ that are always orientable}
\label{sec:3d-always-orientable}

The fact that not all 3d meshes can be consistently oriented of course does
not rule out that there may be important subclasses of 3d meshes that can in
fact be oriented. Indeed, we will show these statements in the
following subsections:
\begin{enumerate}
\item Refining a non-orientable mesh uniformly yields a mesh that is
  consistently orientable.
\item Three-dimensional meshes generated by extruding a
  two-dimensional mesh into a third direction are always orientable.
\item Hexahedral meshes that result from subdividing each tetrahedron of a
  tetrahedral mesh into hexahedra are always orientable.
\end{enumerate}
We will discuss these three cases in the following.

\subsection{Meshes originating from refinement of non-orientable meshes} 

If one is given a three-dimensional mesh with a sheet of parallel edges that
are not consistently orientable, then it turns out that we can generate 
an orientable mesh by subdividing all the cells along this sheet. To
understand this intuitively, remember that a 
non-orientable surface somehow connects to itself ``with a twist of
$180^\circ$''. If we refine all the cells along its way, we duplicate this
sheet and replace it by one that ``goes around twice before connecting to
itself again''. The twist is thus $360^\circ$ and the resulting sheet is
orientable. This is illustrated in Fig.~\ref{fig:doubling} for the two
examples discussed in Section~\ref{sec:3d-always}.

\begin{figure}[tbp]
  \begin{center}
    \phantom{.}
    \hfill
    \includegraphics[width=.4\textwidth]{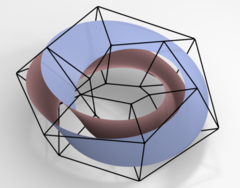}
    \hfill
    \includegraphics[width=.4\textwidth]{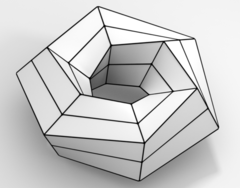}\\
    \hfill
    \phantom{.}
    \\
    \includegraphics[width=.4\textwidth]{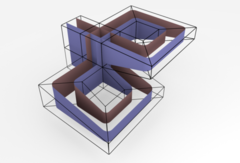}
    \caption{\it Illustration that splitting a non-orientable sheet by
      subdividing each edge along the sheet into two, yields a single
      orientable sheet. Top left: The sheet that results from subdividing
      one of the two M{\"o}bius strips in Fig.~\ref{fig:torus}. Top right: The mesh one obtains by refining all edges along both of
      the non-orientable sheets in Fig.~\ref{fig:torus}. This mesh is
      orientable. Bottom: The
      same for the example shown in Fig.~\ref{fig:rainers-example}.}
    \label{fig:doubling}
  \end{center}
\end{figure}

To make this more formal, let us consider an equivalence class
$\parallelset_{G_\tria}(e)$ of parallel edges which is not consistently orientable.
Here, we use the index to indicate that this is a set of edges in the graph
$G_\tria$. Now split each cell that the sheet associated with $\parallelset_{G_\tria}(e)$
intersects, into 2, 4, or 8 new cells (depending on whether the sheet
intersects the cell once, twice, or three times, i.e., whether one, two, or
all three of the sets of parallel edges within this cell are part of
$\parallelset_{G_\tria}(e)$). We subdivide in the
directions orthogonal to the sheet, see
Figure~\ref{fig:refinement-cases}. This removes the affected edges
from the graph, and replaces them by 8, 24, or 54 new ``child'' edges. With this
refined mesh is associated a new graph $G'$. Then the following results hold:

\begin{figure}[tbp]
  \begin{center}
    \phantom{.}
    \hfill
    \includegraphics[width=.24\textwidth]{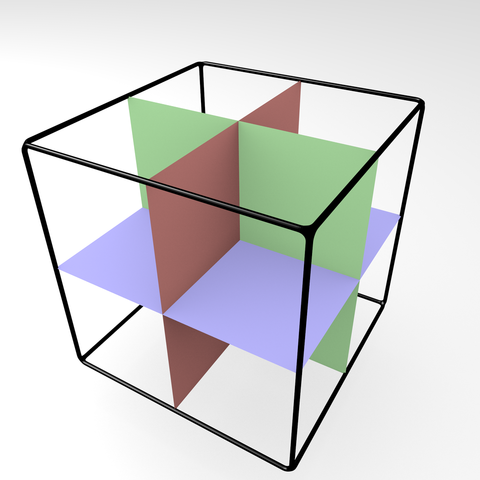}
    \hfill
    \includegraphics[width=.24\textwidth]{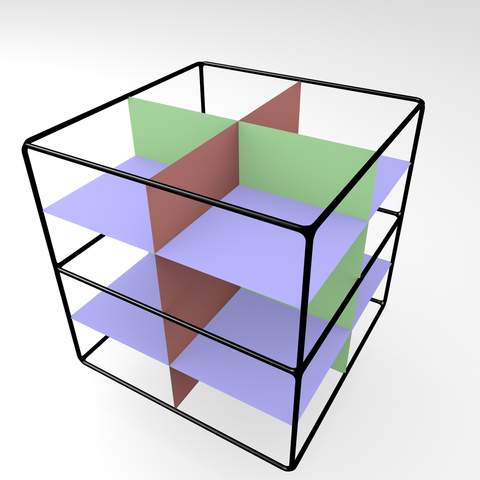}
    \hfill
    \includegraphics[width=.24\textwidth]{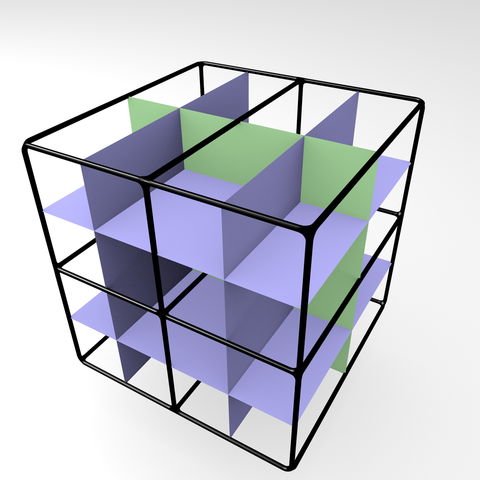}
    \hfill
    \includegraphics[width=.24\textwidth]{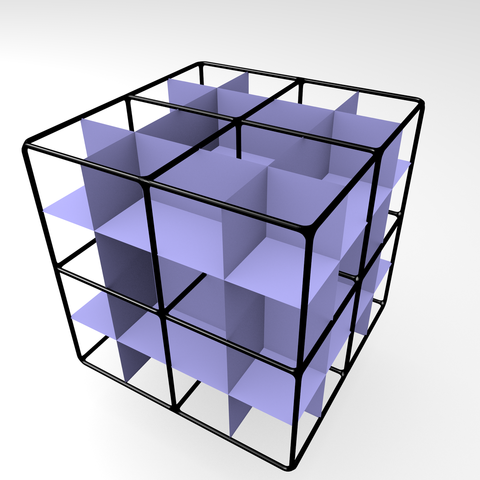}
    \hfill
    \phantom{.}
    \caption{\it Illustration of splitting a cell (left) along the purple
      connecting sheet (left center), as well as along the purple
      sheet in cases where it runs through the same cell twice (center right) or
      three times (right).}
    \label{fig:refinement-cases}
  \end{center}
\end{figure}

\begin{lemma}
  \label{lemma:refinement}
  Let $\Pi_{G_\tria}(e)$ be a set of parallel edges that are not consistently
  orientable, and $\Pi_{G'}(e')$, $\Pi_{G'}(e'')$ the sets of parallel edges
  in the refined graph $G'$ associated with the ``children'' $e',e''$ of
  $e$. Then $\Pi_{G'}(e')=\Pi_{G'}(e'')=:\Pi$. Furthermore, $\Pi$ is orientable.
\end{lemma}

\begin{proof}
  By refining each cell with edges in $\Pi_{G_\tria}(e)$ along the non-orientable
  sheet, we replace each edge $\varepsilon_i\in\Pi_{G_\tria}(e)$ by two edges
  $\varepsilon_i',\varepsilon_i''$. Let us select an arbitrarily chosen
  element from $\Pi_{G_\tria}(e)$, and denote it for simplicity by the symbol $e$. Its
  children, after refinement, $e',e''$ are edges in the refined graph $G'$,
  and their respective sets of parallel edges are
  $\Pi_{G'}(e'),\Pi_{G'}(e'')$. It is easy to see that each of the two
  children $\varepsilon',\varepsilon''$ of each edge $\varepsilon\in\Pi_{G_\tria}(e)$
  must be in either $\Pi_{G'}(e')$ or $\Pi_{G'}(e'')$.

  We first show that these sets are equal. Assume that
  $\Pi_{G'}(e')\neq\Pi_{G'}(e'')$. Then we could find a unique direction for
  each edge $\varepsilon\in\Pi_{G_\tria}(e)$ by assigning it the direction
  ``from the $\Pi_{G'}(e')$ side to the $\Pi_{G'}(e'')$ side.'' This, however,
  would be a consistent orientation of the edges in $\Pi_{G_\tria}(e)$, in
  contradiction to the assumption that this set is non-orientable. Thus,
  $\Pi_{G'}(e')=\Pi_{G'}(e'')$. Intuitively, this means that splitting the
  single sheet associated with $\Pi_{G_\tria}(e)$ still yields a single sheet.
  
  The second step is to show that the single set
  $\Pi=\Pi_{G'}(e')=\Pi_{G'}(e'')$ is orientable. We do this purely locally:
  choose for the two children $\varepsilon',\varepsilon''$ of an edge
  $\varepsilon\in\Pi_{G_\tria}(e)$ the direction away from their common node
  (``centrifugal direction'': the direction vectors of the child edges are
  ``rooted'' in the original sheet through $\Pi_{G_\tria}(e)$). This orientation of the children of
  $\Pi_{G_\tria}(e)$ is consistent with our convention for every one of the affected child cells,
  and, consequently, for all cells which the sheet for $\Pi_{G_\tria}(e)$
  intersects. Intuitively, this can be interpreted as follows: while
  the sheet has no associated normal direction, the direction ``away from the
  sheet'' exists on both sides.
\end{proof}

As can be seen in Fig.~\ref{fig:refinement-cases}, refinement of a cell along one sheet
also adds new edges to other sets of parallel edges. Depending on whether the
sheet intersects a cell once, twice, or three times, refinement adds 4, 5, or
no new edges to other parallel sets. However, we can show that refining the
cells along one non-orientable sheet does not render a previously orientable,
other sheet non-orientable:

\begin{lemma}
  \label{lemma:other-sets}
  Let $\Pi_{G_\tria}(e)$ be a set of parallel edges that are refined to make it
  orientable. Let $\Pi_{G_\tria}(e')$ be a set of parallel edges to which edges are
  added by this refinement step. Then, if $\Pi_{G_\tria}(e')$ was orientable before
  the refinement step, then it is also after the refinement step.
\end{lemma}
\begin{proof}
  If the edges in $\Pi_{G_\tria}(e')$ were already orientable in $G_\tria$, then they must
  have parallel directions in the cell shown in Fig.~\ref{fig:refinement-cases}. If
  we assign the same, parallel direction to the newly added edges in $G'$ that
  belong to this set of parallel edges in $G'$, then the resulting directions
  are consistent in the child cells as well. Because the remainder of the
  cells intersected by $\Pi_{G_\tria}(e')$ are not affected, the locally consistent
  edge orientations in the child cells implies that $\Pi_{G'}(e')$ remains
  orientable.

  On the other hand, if $\Pi_{G_\tria}(e')$ was not orientable, then this fact is
  unaltered by the addition of new edges.
\end{proof}

With these lemmas, we can state the final result of this section:

\begin{theorem}
  Let $G_\tria$ be the graph associated with a subdivision of a domain into
  hexahedra. If it is not orientable, then we can generate a new graph $G'$ by
  refining all cells exactly once along the sheets associated with those sets
  of parallel edges that are not orientable.
\end{theorem}

\begin{proof}
  The theorem follows from the previous two lemmas. As was shown in
  Lemma~\ref{lemma:refinement}, we can convert a non-orientable set of
  parallel edges into an orientable one. In Lemma~\ref{lemma:other-sets}, we
  showed that this does not create new non-orientable sets. Since the original
  graph can only have finitely many non-orientable sets of parallel edges, we decrease this
  number by one in each refinement step and thus obtain an orientable graph in
  a finite number of steps.
\end{proof}

Thus, even though there are three-dimensional meshes that cannot be
oriented according to our convention, above theorem shows that there is a
simple and inexpensive remedy for these cases. While the meshes resulting from
such anisotropic refinement have a worse aspect ratio, one can
simply refine \textit{all} hexahedra uniformly into eight 
children to obtain an orientable mesh with the same aspect ratio of cells as
the original one. This corresponds to refining the cells intersected by \textit{all}
sheets associated with sets of parallel edges, not only those corresponding to
non-orientable ones. Obviously, this mesh is also orientable: in addition to
the originally non-orientable parallel sets, we now also refine the originally
orientable parallel sets, which yields two distinct sets of parallel edges
$\Pi_{G'}(e')$ and $\Pi_{G'}(e'')$ that can independently be oriented, one on
``this'' side of the original sheet and one on ``that'' side of it (these
sides are distinct because the sheet associated with $\Pi_{G_\tria}(e)$ was
orientable).

\begin{remark}
  The observations of this section also point out an important
  optimization in practice for codes that create finer
  meshes by subdividing the cells of an existing mesh. In such cases, it is not
  necessary to re-run the mesh orientation algorithm on the finer mesh
  with four (2d) or eight (3d) times as many cells. Rather, if the
  original mesh was already consistently oriented, then the new mesh
  will be consistently oriented by simply choosing the directions of
  the refined edges to be the same as those of their parent edge.
\end{remark}

\subsection{Extruded meshes} 
\label{sec:3d-extrude}

An important class of meshes consists of those that start with a
two-dimensional quadrilateral mesh and then ``extrude'' it into a third direction by
replicating it one or more times and connecting the vertices of the
original mesh and its replicas in this third direction. Indeed, the
mesh shown at the top of Fig.~\ref{fig:torus} is such a mesh: the
sequence of quadrilaterals at the bottom has been replicated to the
top, and each pair of original and replicated vertices are connected
by a new edge.

Extruded meshes are often used for ``thin'' domains. The
technique is obviously also applicable if the original two-dimensional
mesh lived on an (orientable) manifold such as the bottom surface of
the object we want to mesh. The extrusion also need not necessarily be
in a perpendicular direction, nor do the replicas have to be parallel
to the original mesh. That said, for the purposes of orienting edges,
such geometric considerations are immaterial. 

For extruded meshes, we can state the following result:
\begin{theorem}
  \label{theorem:extruded-always}
  Let $\{K\}$ be a hexahedral mesh obtained by extruding a
  quadrilateral mesh defined on a two-dimensional, orientable manifold
  in a third direction. Then the
  edges of this mesh are consistently orientable.
\end{theorem}

\begin{proof}
  To understand why this is the case, recall that \textit{all}
  two-dimensional quadrilateral meshes on such manifolds are orientable. In other words,
  in the original mesh, opposite edges can already be chosen to be
  parallel, and this is also the case for its replicas. 

  Next, consider the sets of parallel edges we may consider. Obviously, the edges
  of the original mesh and its replicas are parallel, and we can
  consistently orient them if we choose edge directions of the
  original mesh and its replicas the same. An alternative viewpoint is
  that the sheets that connects these sets of parallel edges are the
  ``extruded'' version of the line that connected these parallel edges
  in the two-dimensional mesh, and the orientability of the
  one-dimensional line then extends to the orientability of the
  corresponding sheet to which it gives rise.

  The only additional sets of parallel edges are the ones that
  connect the original mesh and its first replica, and then each
  replica with the next. The corresponding sheets can be thought of as
  copies of the two-dimensional domain spanned by the original mesh,
  located halfway between the replicas. Each of these sheets is
  obviously orientable: The edges of each of these new sets are
  independent of each other, and each class can be consistently
  oriented, for example by always choosing the direction from the
  original to the first replica, and from each replica to the next
  (i.e., an ``upward'' direction).
\end{proof}

\subsection{Meshes originating from tetrahedralizations}

The examples in Section~\ref{sec:3d-always} show that there is no
\textit{topological} characterization of those 
domains for which we can always find a consistent orientation of
edges. Rather, it is a question of \textit{how} that domain
was subdivided into cells. Indeed, we can show that for the very
general class of domains that can be subdivided into
tetrahedra, there also exists a subdivision into hexahedra with a consistent
orientation of edges:

\begin{theorem}
  \label{theorem:tet-to-hex}
  Let $\{T\}$ be a given subdivision of a domain $\Omega$ into a set of
  tetrahedra so that two distinct tetrahedra either have no intersection,
  share a common vertex, a complete edge, or a complete face. Divide each
  tetrahedron $T$ into four hexahedra, $K(T)_0,\ldots, K(T)_3$, by using the
  vertices, edge midpoints, face midpoints, and cell midpoint of $T$ as
  vertices for the hexahedra.  Then the edges of the mesh consisting of the
  union of these hexahedra, $\bigcup_{T, 0\le i< 4} K_i(T)$, are consistently
  orientable.
\end{theorem}

Before we prove this claim, we note that the resulting hexahedra are
in almost all cases not close to the optimal cube shape. Also, almost
all vertices in the resulting mesh have a 
number of cells meeting at this vertex that deviates from the optimal
number of eight (encountered, for example, in a regular subdivision into
cubes). These meshes are therefore hardly optimal for finite element
computations. However, given the complexity of generating hexahedral
meshes, until recently many mesh generators used this approach to
generate an initial hexahedral mesh. (For example, the widely used
open source mesh generator gmsh \cite{gmsh} can use this
approach to generate hexahedral meshes.)

\begin{figure}[tbp]
  \begin{center}
    \phantom{.}
    \hfill
    \includegraphics[width=.22\textwidth]{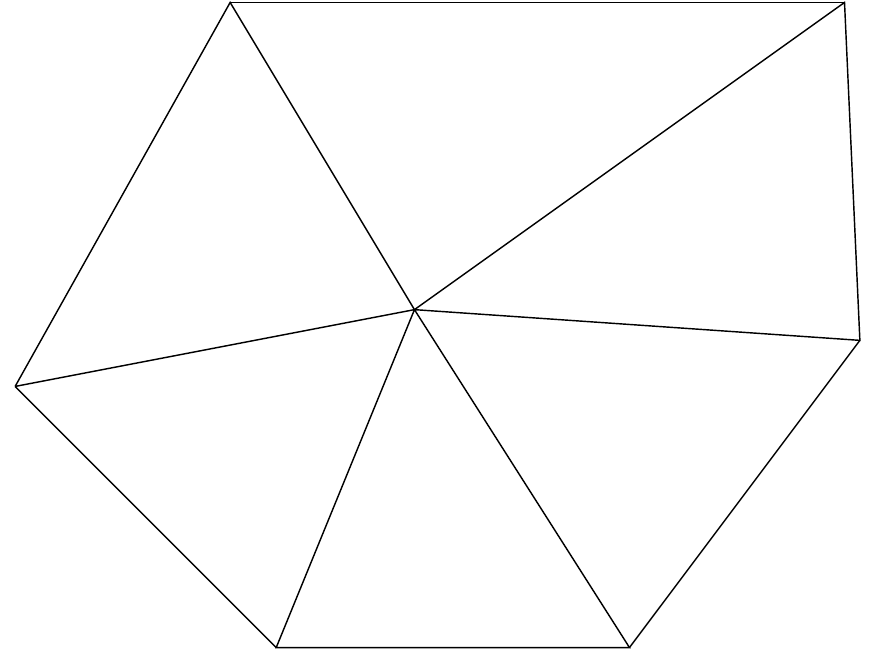}
    \hfill
    \includegraphics[width=.22\textwidth]{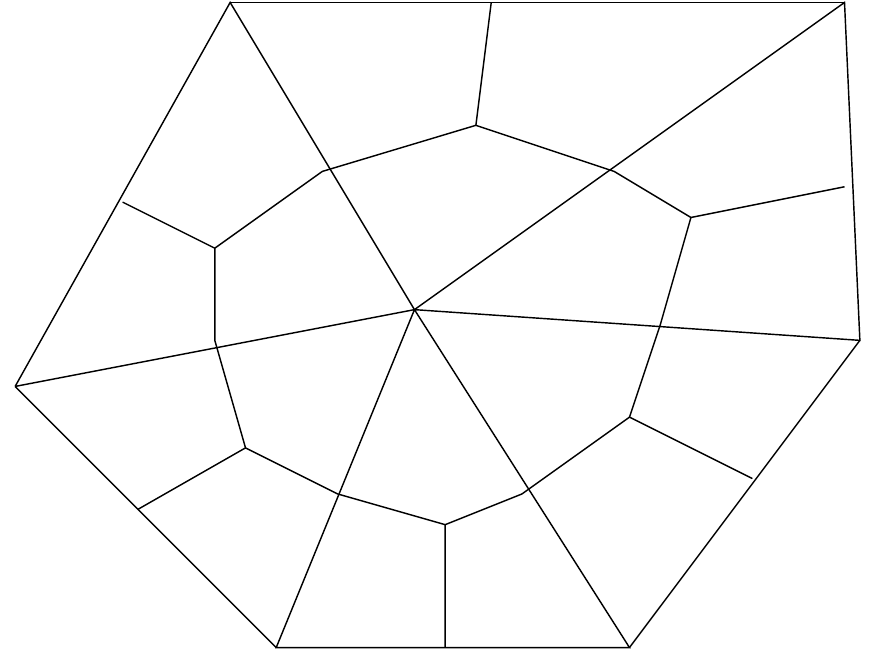}
    \hfill
    \phantom{.}
    \\[12pt]
    \includegraphics[width=.22\textwidth]{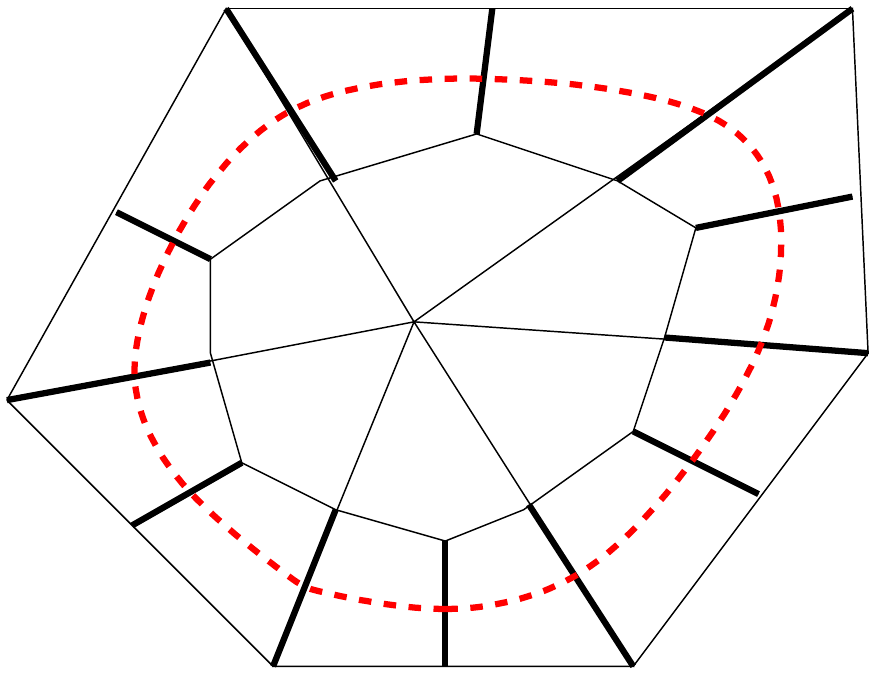}
    \hfill
    \includegraphics[width=.22\textwidth]{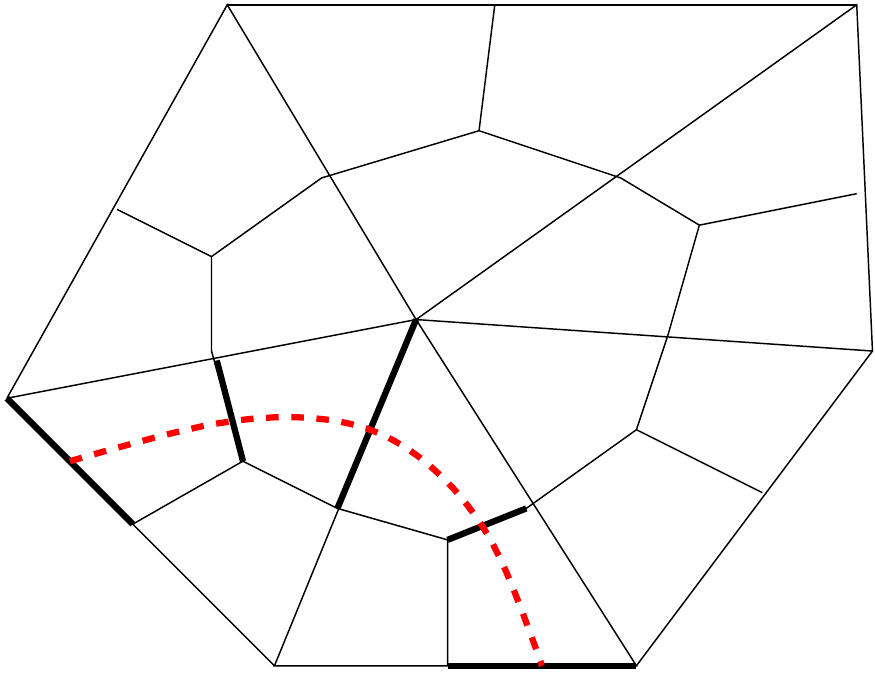}
    \hfill
    \includegraphics[width=.22\textwidth]{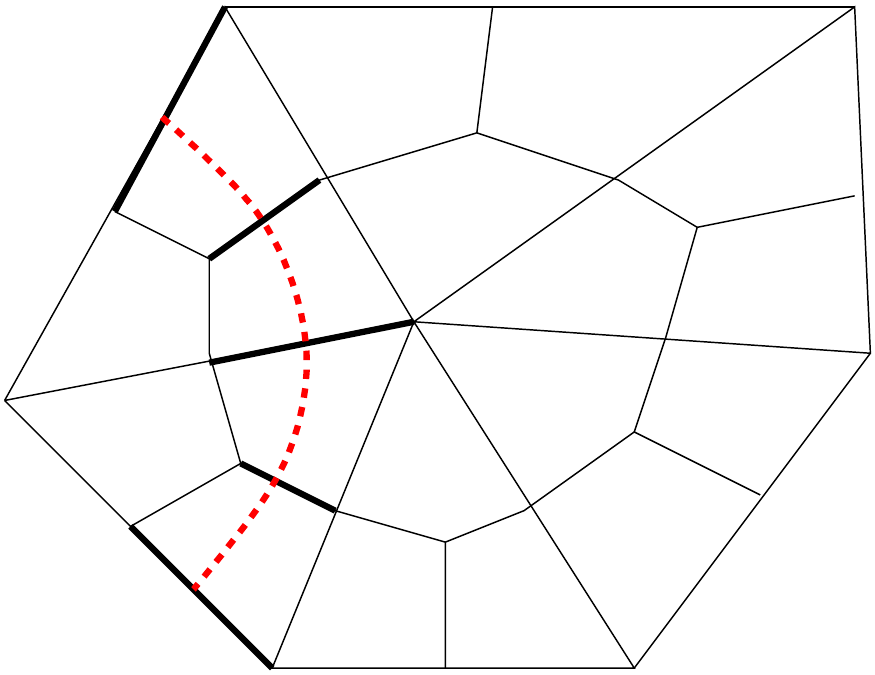}
    \caption{\it Construction in $d=2$ for the proof of
      Theorem~\ref{theorem:tet-to-hex}. Top: Part of a subdivision into
      triangles before and after cutting each triangle into three
      quadrilaterals (top row). Bottom: Three sets of parallel edges and their
      connecting line (bottom row).}
    \label{fig:theorem-2-2d}
  \end{center}
\end{figure}

\begin{proof}[of Theorem \ref{theorem:tet-to-hex}]
  In order to explain the proof in $d=3$, it is instructive to first
  consider a similar case in $d=2$, where we would cut all cells of a triangular mesh
  into three quadrilaterals, see
  Fig.~\ref{fig:theorem-2-2d}. From the edges of these
  quadrilaterals, we then generate independent sets of parallel
  edges; the bottom row of the figure shows three of these sets
  along with their connecting line. 

  Importantly,
  each of these lines forms a closed, non-self-intersecting loop
  around one of the original vertices of the triangles, cutting
  through all the cells in the second layer of quadrilaterals around
  this vertex (unless, of course, the vertex is at the boundary of the
  domain, in which case the line is not closed but starts and ends at
  the boundary). Indeed, the edges of each quadrilateral are
  part of the loops for two vertices of the triangle
  from which it arose. Each
  vertex thus gives rise to at least one set of parallel
  edges in the subdivision into quadrilaterals. These can, of course, all be
  oriented in $d=2$. (It may give rise to more than one such set if
  two parts of the domain touch at a vertex, i.e., if there are two
  sets of cells adjacent to a vertex that are not mutual face neighbors.)

  It is easy to generalize these considerations to the case $d=3$ (though we
  have not found good ways to visualize hexahedral meshes resulting from
  even a small collection of tetrahedra forming an unstructured mesh).
  First, we note that each original interior vertex
  now induces a set of parallel edges that is connected by a closed
  non-self-intersecting sheet around the vertex. Since it is homeomorphic to the
  surface of the unit sphere, it is of course orientable, and so are the edges
  of this parallel set. For original vertices on the boundary, the sheet is not
  closed, but it is obviously still orientable.
\end{proof}

\begin{includefaces}

\section{Extension to face orientations for 3d meshes}
\label{sec:3d-faces}

In 3d, one not only has the problem of aligning the coordinate system of
one-dimensional edges with that of cells, but a similar problem also appears
with the coordinate system of two-dimensional faces to which one would like to
associate data (such as degrees of freedom or their numeric values). This then
raises similar questions as in the edge case: can we assign an orientation --
in the simplest case a normal vector -- to each face that can uniquely be determined
from each of the adjacent cells simply from its position within this cell?

The answer to this first question is an unconditional ``yes'', by a similar argument
as the one we used for edge directions in 2d. Here, we first have to collect
the set of parallel \textit{faces}, denoted by $\parallelset(f)$, of a
particular face $f$. These faces are constructed by always hopping from one
face to the one opposite of it in a cell, so we can draw a line through all of
them. The second step is to give all of them a certain direction. In 2d, we
used that each line is orientable, so we could use the direction pointing from
one side of the line to the other as the direction for the edges it
intersects. Here, we simply use the ``forward'' (or ``backward'') direction as
we move along the line connecting the faces in $\parallelset(f)$, to provide
each face with a normal direction that matches that of all parallel faces
within its two adjacent cells. By construction, this line cannot split, so
even if it is a closed or self-intersecting line, we can always define such a
direction; thus there is a consistent orientation of all faces in
$\parallelset(f)$ for all subdivisions of domains into hexahedra, and an
algorithm is easily derived that does this in a time linear in the number of
faces or cells.

In reality, however, the orientation of the face's coordinate system is not
only described by a single bit such as whether its normal vector points into
or out of a cell. Rather, the coordinate system of the face may also be
rotated by $0^\circ, 90^\circ, 180^\circ$, or $270^\circ$ against the
coordinate system of the cell when restricted to that face. An alternative way
of describing these four rotations is to state the vertex at which the
coordinate system originates. Defining properties such as these does
not naturally fit into the graph theoretical context we have so far
used, and we will therefore use some concepts of discrete geometry in
the following.

To concisely define what properties we would like to have in the
finite element context, let us introduce two
definitions that assume that a mesh is given by a collection of
vertices $V$, linear edges $E$, quadrilateral faces $F$, and
hexahedral cells ${\tria}$ (where obviously each of faces bound one
or two cells, and each edge bounds at least one face). We then want to
associate with each cell $K\in {\tria}$ a right handed coordinate
system and denote $K^+$, and similarly for faces and edges which we
will then denote $f^+$ and $e^+$ for $f\in F, e\in E$. For each edge,
there are two possible ways of defining $e^+$, namely the two edge
orientations. For each face $f$, the previous paragraphs have shown that
there are 8 possibilities to define $f^+$. One can easily verify that
for each cell $K$, there are 24 possible $K^+$. One can then define
consistent orientations as follows:

\begin{convention}
  An oriented cell $K^+$ and its bounding faces $f^+(K)$ and bounding
  edges $e^+(K)$ are called \textit{consistently edge and face oriented} if the
  coordinate systems defined on the edges and faces bounding $K$ are
  simply the restrictions of the cell's coordinate system to the edge
  or face.
\end{convention}

\begin{figure}[tbp]
  \begin{center}
    \phantom{.}
    \hfill
    \includegraphics[width=.3\textwidth]{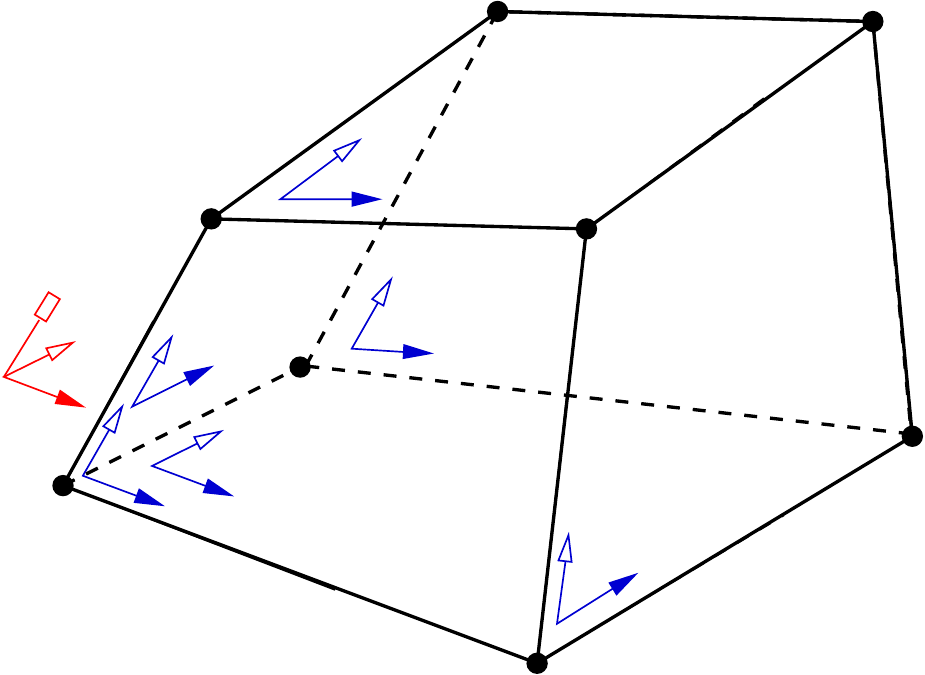}
    \hfill
    \hfill
    \includegraphics[width=.3\textwidth]{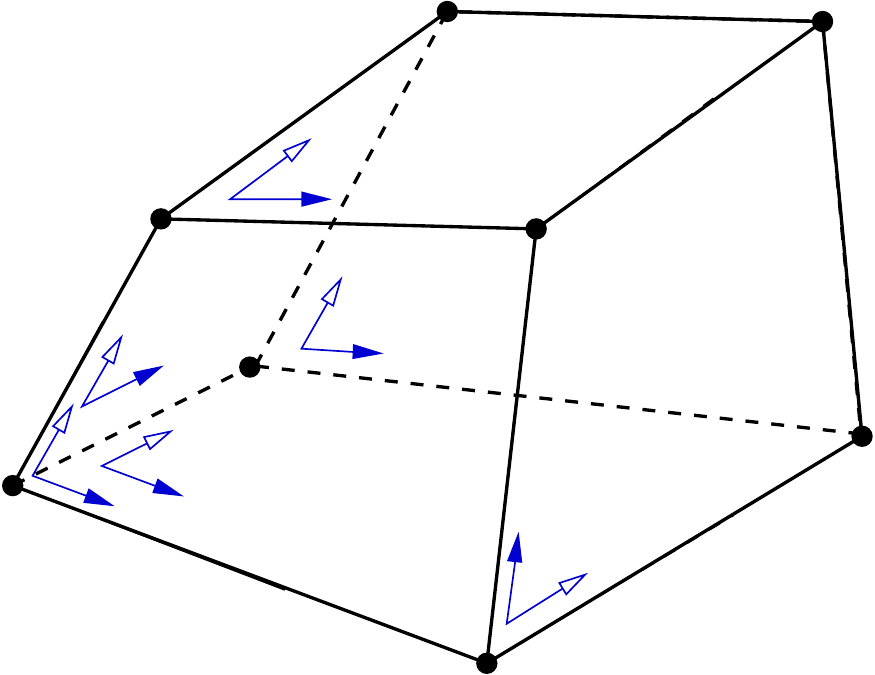}
    \hfill
    \includegraphics[width=.3\textwidth]{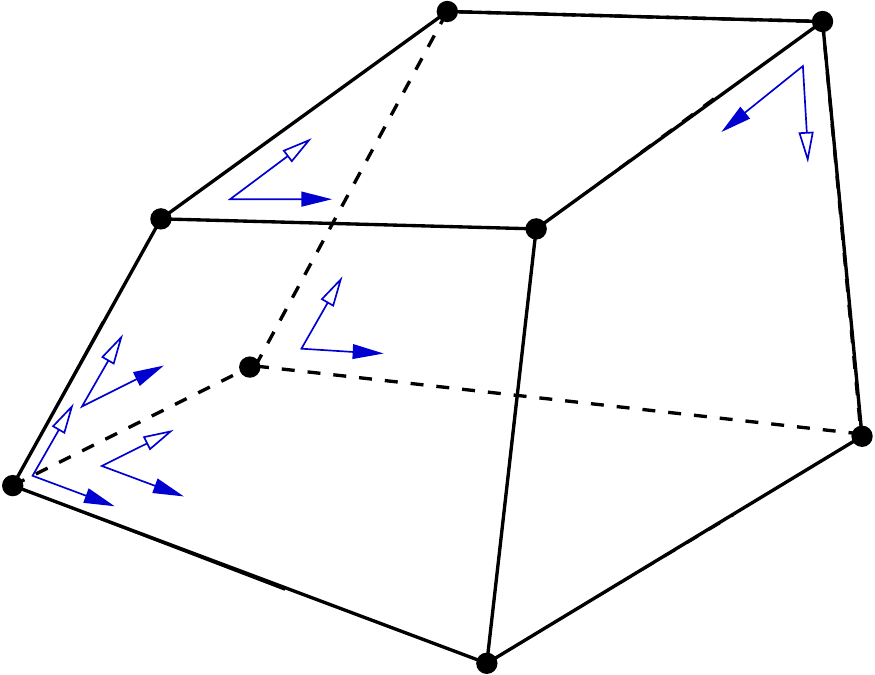}
    \hfill
    \phantom{.}
    \caption{\it Left: Illustration of consistently face oriented cell
    in 3d. For each face, a filled arrow indicates the local $x$-axis
    and an open arrow the local $y$ axis. The coordinate system of the
    cell is shown in red to the left, with the local $z$ axis using an
    open box as arrow. The cell would also be
    consistently edge oriented if edges were oriented as shown in the
    right panel of Fig.~\ref{fig:convention}.
    Center: A cell that could still be consistently edge oriented but
    is not consistently face oriented because the right face has an
    inverted coordinate system.
    Right: A cell where the right face's coordinate system has its
    origin at the wrong location. For this cell, edge orientations can
    not be chosen in such a way that they would be consistent with
    both the top and right face, for example.}
    \label{fig:3d-face-edge-orientations}
  \end{center}
\end{figure}

An example of a consistently oriented cell, and two cells that are
not consistently oriented are shown in
Fig.~\ref{fig:3d-face-edge-orientations}. Note that for a consistently
edge and face oriented cell, the coordinate systems defined on the
edges are naturally not only the restrictions of the cell coordinate
systems to this edge, but also the restrictions of the face coordinate
systems to this edge.

This convention is a natural
generalization of the two-dimensional edge orientations we have used
in Section~\ref{sec:2d}. There, we only considered the coordinate
systems on the edge (described by the edge orientation) and did not
mention a coordinate system on cells, but it is easy to see that we
can choose a unique right-handed coordinate system originating at the
one vertex with two outgoing edges, as long as opposite edges are
parallel. Given this local consistency, we can define globally
consistently oriented meshes:

\begin{convention}
  A mesh with associated coordinate systems,
  $(V,E^+,F^+,{\tria}^+)$, is called \textit{consistently edge and
    face oriented} if each cell and its bounding faces and edges are
  consistently oriented.
\end{convention}

Clearly, this is a stronger condition than just the edge orientations
considered in Section~\ref{sec:3d} and every mesh that is consistently
edge and face oriented is also consistently edge oriented following
Definition~\ref{conv:2} (as well as consistently face oriented, using
the definition earlier in this section). It is therefore clear that at least those
meshes that can not be consistently edge oriented can also not be
consistently edge and face oriented. On the other hand, we have shown
in Section~\ref{sec:3d-always-orientable} that important classes of
meshes are always edge orientable, and one may hope
that these are in fact also face orientable.

Unfortunately, this turns out to be not true, even for the otherwise
relatively well-behaved class of extruded meshes for which we showed in
Section~\ref{sec:3d-extrude} that their edges can always be oriented.
This can be shown for a simple counter-example involving only three
cells and that is shown in Fig.~\ref{fig:3d-face-edge-counter}.

\textbf{ACTUALLY, THIS APPEARS TO REQUIRE MORE THOUGHT. I'M NO LONGER
  CLEAR ON WHAT IS INTENDED OR IMPLEMENTED, AND THEY MAY BE DIFFERENT.}

\begin{figure}[tbp]
  \begin{center}
    \phantom{.}
    \hfill
    \includegraphics[width=.3\textwidth]{graphics/convention_3d_faces_counter}
    \hfill
    \phantom{.}
    \caption{\it xxx.}
    \label{fig:3d-face-edge-counter}
  \end{center}
\end{figure}

This example shows that there are simple examples of meshes that
cannot be consistently edge and face oriented at the same time. It is
conceivable that there is a simple characterization for which meshes
consistent orientations can be found, in the same way as we showed in
Section~\ref{sec:3d} that meshes can be consistently oriented if no
set of parallel edges is connected by a non-orientable surface. We
have not attempted to find such a characterization here since the
counterexample showed that the subset of consistently edge and face
orientable meshes is too small for practical purposes.

\end{includefaces}

\section{Conclusions}
\label{sec:conclusions}

Finite element codes can be made significantly simpler if they can
make assumptions about the relative orientations of coordinate
systems defined on cells, edges, and faces. If such assumptions always hold,
then this reduces the number
of cases one has to implement and, consequently, the potential for
bugs. In this paper, we have described a way to orient edges and the cells
they bound, and shown that not only
is it possible to choose edge directions consistently with regard to
this convention for two-dimensional quadrilateral meshes, but also
that there is an efficient algorithm to find such edge orientations.

On the other hand, it is not always possible to orient edges of
three-dimensional hexahedral meshes according to the three dimensional
generalization of this convention. The obvious generalization of
our two-dimensional algorithm is able to detect these cases, again in
optimal complexity, but the result implies that codes dealing
with hexahedral meshes necessarily have to store flags for each of the
edges of each cell that indicate the orientation of that edge relative
to the coordinate system of the cell. This is not a significant
overhead in terms of memory and possibly not in terms of algorithmic
complexity. Nevertheless, in actual practice, this has turned out to be an endless
source of frustration and bugs in \dealii{} as the cases where edge
orientations are relevant are restricted to the use of higher order
elements, as well as complex and three-dimensional geometries. In case of bugs,
methods generally converge but at suboptimal orders. Consequently,
debugging such cases and detecting where in the interplay of geometry,
mappings, degrees of freedoms, shape functions, and quadrature the bug
resides has proven to be a very significant challenge. This experience also
supports our claim that being able to enforce a convention in the
two-dimensional case almost certainly saved a great deal of development time.

At the same time, this paper at least identifies broad classes of
three-dimensional meshes for which one can always consistently orient
edges, and for which no special treatment of edges is
necessary. Through counter-examples, we have shown that there is no
topological description for which domains do or do not allow for
consistent orientations, but that it is indeed a property of how the
domain is subdivided into cells, and our analysis demonstrates ways by which
meshes can be constructed in ways so that edges can always be oriented
consistently. This analysis can therefore also provide constructive feedback
for the design of mesh generation algorithms.

\begin{includefaces}
In an ideal world, finite element codes would use meshes for which not
only edges but also faces are consistently orientable. As shown in
Section~\ref{sec:3d-faces}, this turns out to be a much more
restrictive condition, and a simple counterexample demonstrates that
we cannot hope for such a property for meshes that arise in realistic
situations.
\end{includefaces}

\paragraph*{Acknowledgments} WB would like to thank J.~M.~Landsberg,
J.-L.~Guermond, and A.~Ern for illuminating discussions.

\appendix

\section{Quadrilateral meshes on two-dimensional
  manifolds}
\label{sec:algebraic-topology}

Lemma~\ref{lemma:existence-for-one-parallel-set} provided the basis
for the proof that we can find consistent orientations of the edges
for all meshes that subdivide a domain $\Omega\subset\R^2$ into
quadrilaterals. Fundamentally, the reason for its truth is the
geometric fact that the piecewise linear curve connecting a set of
parallel edges has a unique ``left'' and ``right'' side.

As mentioned in Remark~\ref{remark:extension-manifold}, this can be
generalized to the connecting lines on orientable, two-dimensional
manifolds. Geometric proofs for this extension are complicated by two
facts: (i) the underlying manifold may not be smooth, for example in
the case of a mesh on the surface of a body with edges and corners
(such as a cube); (ii) we need to be careful how exactly we embed the
curve connecting parallel edges into the manifold. Consequently, it is
easier to avoid the language of geometry altogether. Rather, we will
use the language of combinatorial topology that in its essence only
uses what we are given: the quadrilateral mesh.

In the statement of the following result, we will use the combinatorial
topology definition of what an ``orientable''
manifolds is (see \cite{Lee11,Mun96} and below). This class of manifolds
includes all smooth two-dimensional manifolds that are orientable in
the differential geometry sense \cite{Spi65}, but also (parts of) the
boundaries of domains in $\R^3$ \cite[Chapter VI, Theorem
  7.15]{Bre93}. Furthermore, the manifolds we can consider need not be
naturally embedded in $\R^3$. A special
case of an orientable, two-dimensional manifold is of course the plane
$\R^2$, covering the situation of
Lemma~\ref{lemma:existence-for-one-parallel-set}.
Then the following is
true:

\begin{theorem}
  Let $\tria$ be a quadrilateral mesh with finitely many cells on
  an arbitrary, two-dimensional, orientable manifold, and $G_\tria$
  the associated graph. Then there exists a corresponding directed
  graph $G_\tria^+$ that is consistently oriented.
\end{theorem}

\begin{proof}
The proof is in essence analogous to that of
Theorem~\ref{theorem:2d}. The only piece that has changed is that we
need to provide for the orientability of edges in each of the parallel sets,
i.e., the extension of
Lemma~\ref{lemma:existence-for-one-parallel-set} to orientable,
two-dimensional manifolds. 

To this end, let us construct a
subdivision $\Delta$ of $\tria$ into 2-simplices (triangles) by adding to each
quadrilateral one of the two diagonals. $\Delta$ is then a
simplicial 2-complex.
A 2-complex is called coherently oriented if (i)~the edges of each
2-simplex are either oriented clockwise or counterclockwise, and (ii)~whenever
two 2-simplices share a common edge, the relative 
orientations of this shared edge are complementary (i.e., opposing),
see for example \cite[Chapter 5]{Lee11}. A 2-manifold is called
orientable if every 2-complex on it
can be coherently oriented. Thus, since
$\Delta$ is a 2-complex on a surface for which we have previously
assumed that it is orientable, we may
choose a coherent orientation of all edges of $\Delta$. Given the subdivision of each
cell $K\in\tria$ into two triangles with edges oriented either clockwise or
counterclockwise, this induces an orientation on
$\tria$ where (i) opposite edges in each cell $K$ are oriented
complementarily, and (ii) shared edges of two adjacent cells are
oriented complementarily. One quickly checks that this is independent
of the choice of diagonals. 

We have previously seen that each edge $e$ is part of at most two
cells and has therefore at most two opposite edges. 
For each set of parallel edges, $\parallelset_i\in\pi$, there is then a sequence
of cells connected by the edges $\parallelset_i$ that is either open or that
forms a closed loop. This 
is illustrated in Fig.~\ref{fig:algebraic-topology}. (Cells may
occur more than once in this sequence, but only once for each pair of
opposite edges in this cell.)

\begin{figure}[tbp]
  \begin{center}
    \phantom{.}
    \includegraphics[width=0.28\textwidth,angle=90]{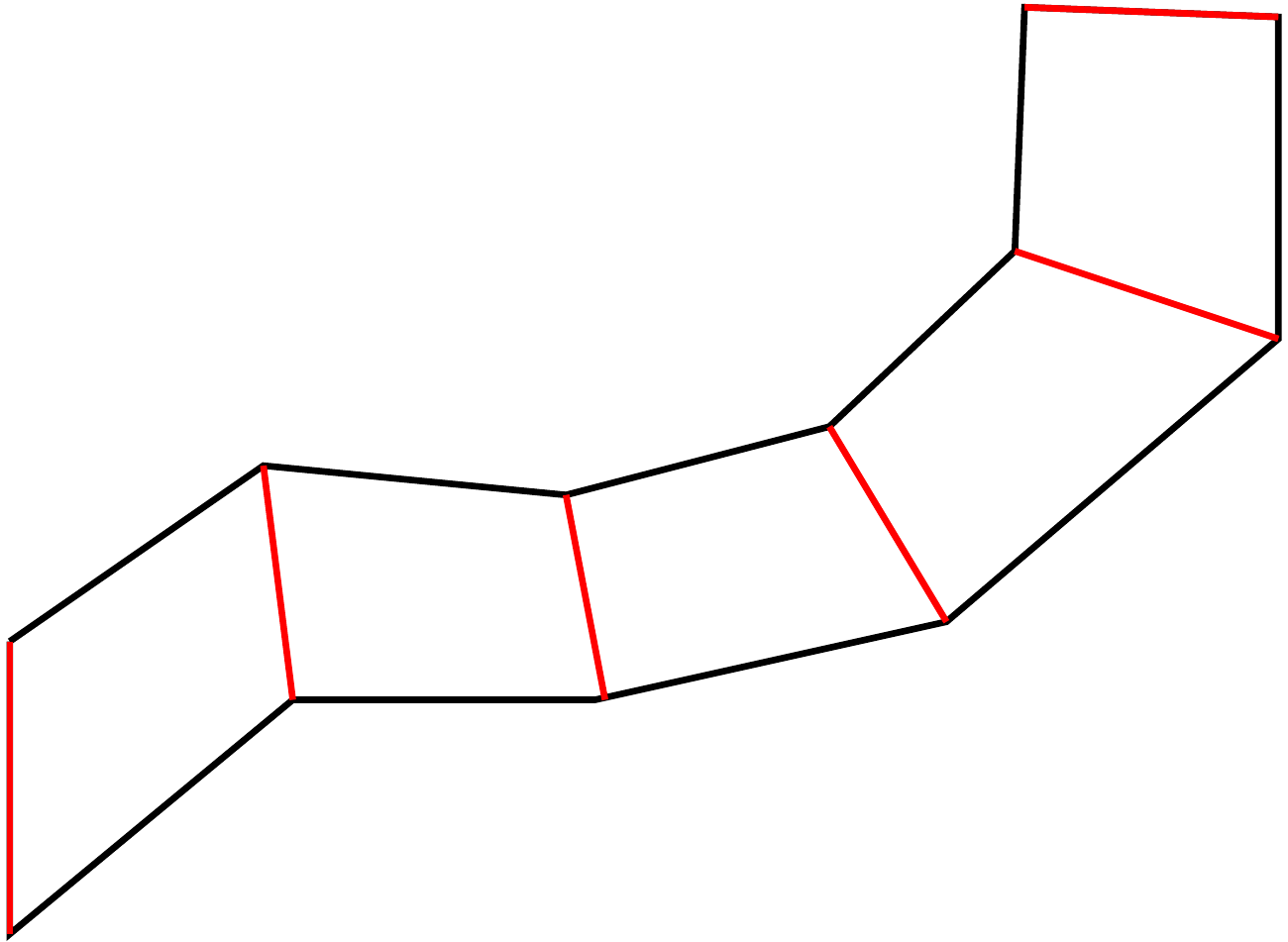}
    \hfill
    \includegraphics[width=0.28\textwidth,angle=90]{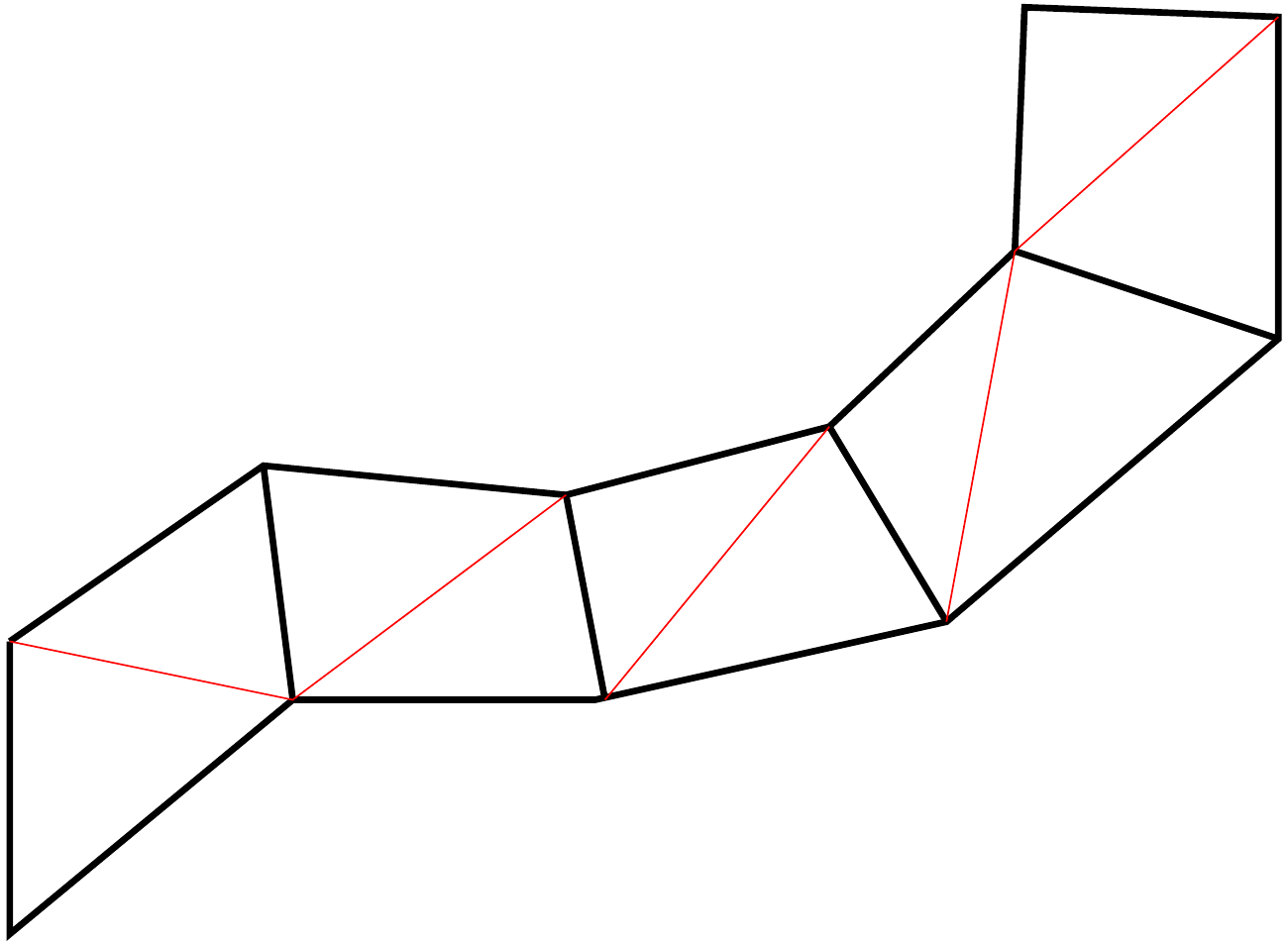}
    \hfill
    \includegraphics[width=0.28\textwidth,angle=90]{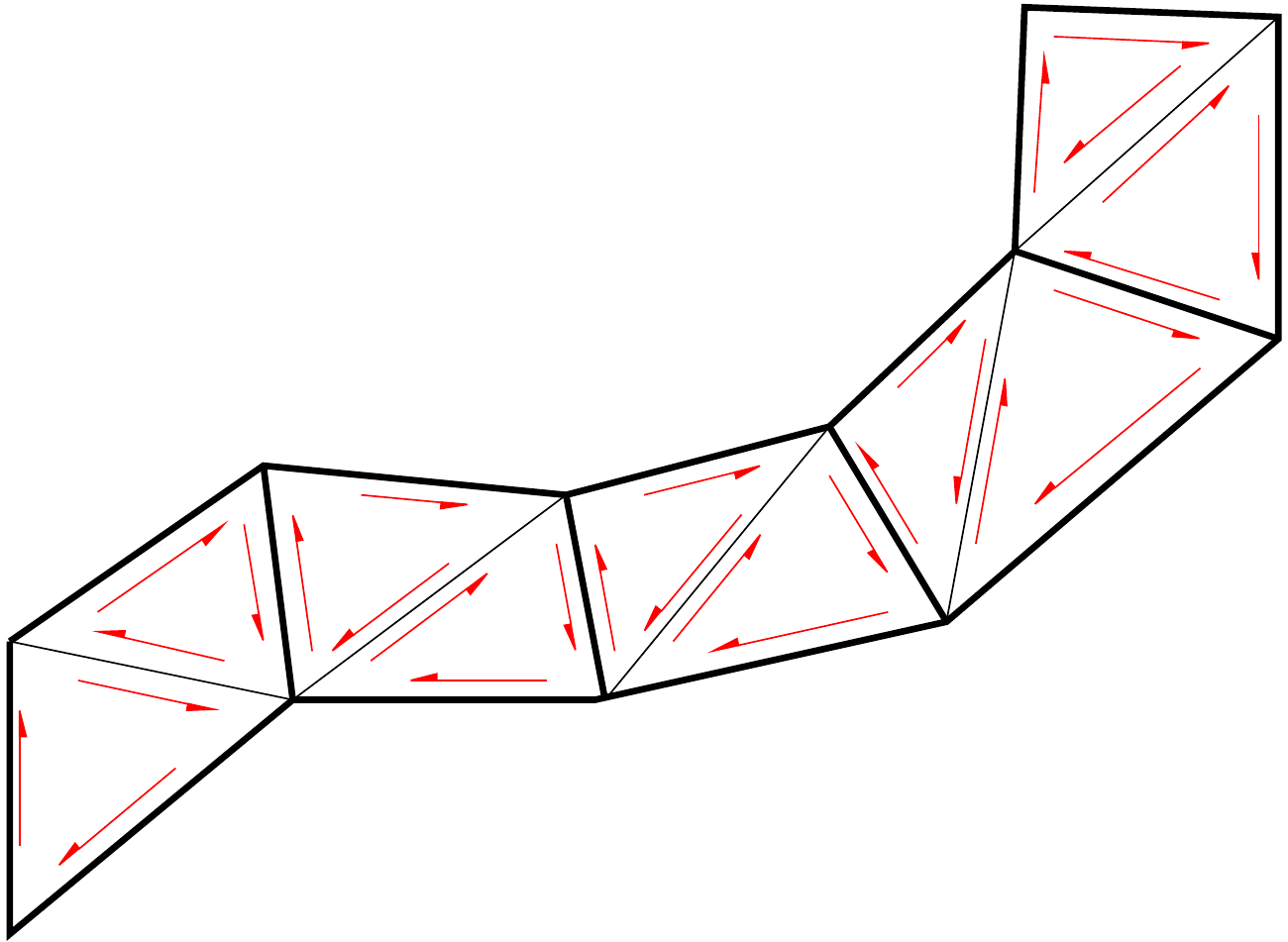}
    \hfill
    \includegraphics[width=0.28\textwidth,angle=90]{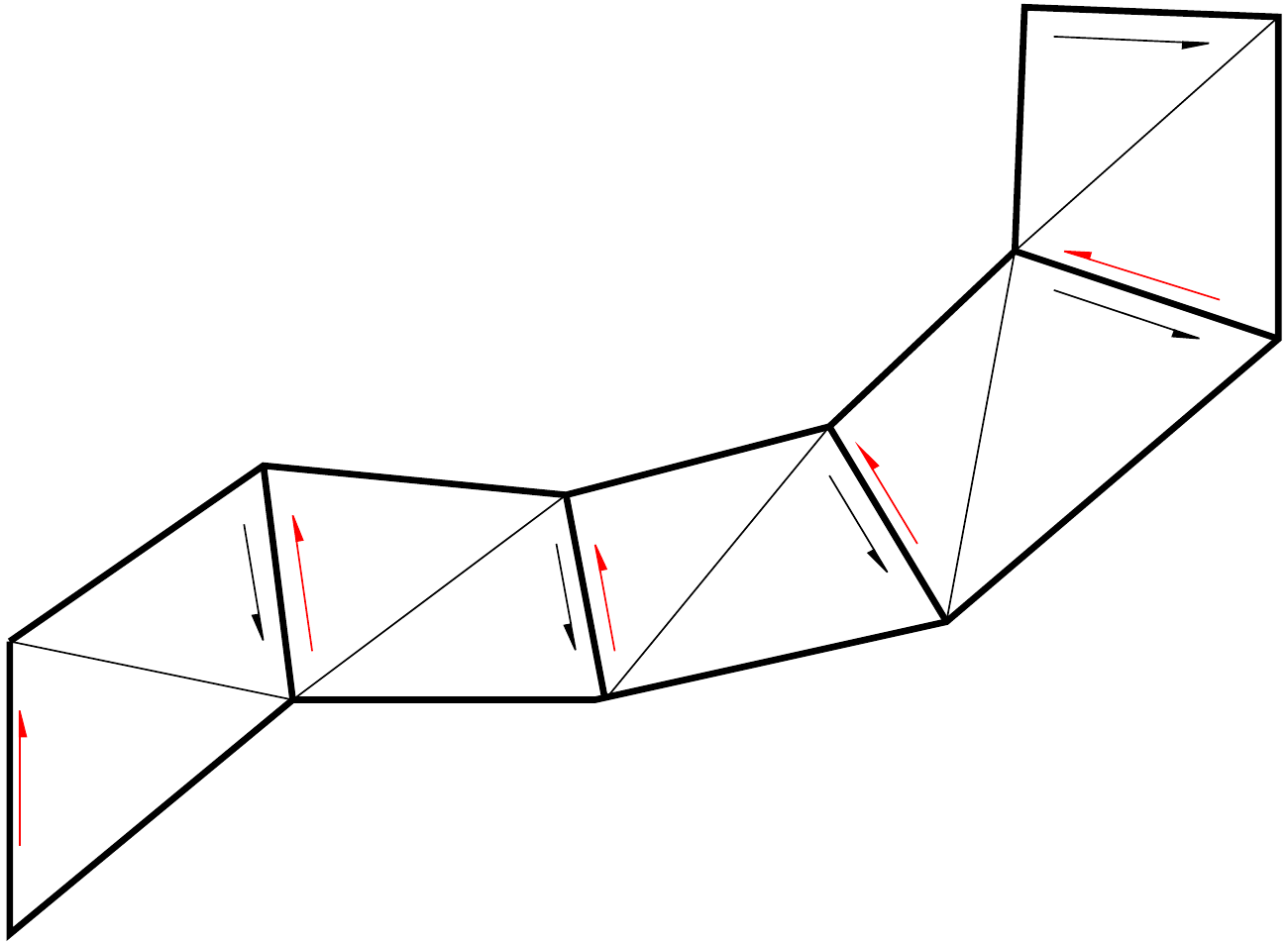}
    \phantom{.}
    \caption{\it Construction of consistent edge orientations via simplicial
      complexes. Left: The part of a mesh consisting of the cells adjacent to the
      edges (marked in red) that form one set of parallel edges. Center left:
      One possible subdivision of quadrilateral 
      cells into triangles by adding a diagonal to each
      cell. The result is a simplicial complex $\Delta$. Center right: Because
      the mesh exists on a 2-manifold that is 
      orientable, it is possible to assign
      clockwise and counter-clockwise orientations to all triangles so that
      they imply complementary directions for each shared
      edge. Right: Choosing every other direction for edges between
      cells, as we walk along the sequence of cells, yields a set of
      consistent edge orientations, i.e., opposite edges in quadrilaterals are
      oriented in a parallel direction.}
    \label{fig:algebraic-topology}
  \end{center}
\end{figure}

Now let us choose an edge
$e\in\parallelset_i$, and let $K$ be one (of possibly two) neighboring cell
of $e$. If the sequence of cells connected by $\parallelset_i$ is open, then
assign to $e$ the direction of this edge as defined in $K$. The edge $e'$
opposite of $e$ in $K$ is then assigned the direction defined in the
neighbor of $K$ beyond $e'$, and so forth in both directions starting
at $e$. It is easy to see that this leads to a consistently oriented
set of edges in the sense discussed in Section~\ref{sec:2d}.

If the sequence of cells connected by $\parallelset_i$ is closed, we need
to ensure that this does not lead to a conflict. If we follow
the sequence of cells, the previous construction chooses
every second encountered edge orientation; because each orientation is
complementary to the previous, choosing every other one leads to a
consistent set of edge orientations.

We then repeat this for all classes $\parallelset_i\in\pi$, thus finishing the
proof.
\end{proof}

\end{document}